\documentclass[preprint,leqno]{elsarticle}

\usepackage{amsmath, amsthm, amssymb}
\usepackage{booktabs}
\usepackage{placeins}
\usepackage{xcolor}
\usepackage{breqn}

\usepackage[breaklinks]{hyperref}
\usepackage[hyphenbreaks]{breakurl}

\usepackage{graphicx}
\usepackage{fix-cm}

\usepackage{amsmath,amssymb,hyperref}
\usepackage{footmisc}
\usepackage{xcolor}
\usepackage{pgfplots}

\theoremstyle{plain}
\theoremstyle{definition}
\newtheorem{theorem}{Theorem}
\newtheorem{proposition}[theorem]{Proposition}
\newtheorem{corollary}[theorem]{Corollary}

\newtheorem{remark}[theorem]{Remark}

\newtheorem{definition}[theorem]{Definition}

\journal{European Journal of Operational Research}

\usepackage[ruled,vlined,linesnumbered]{algorithm2e}
\def\R{\mathbb{R}}
\def\Z{\mathbb{Z}}

\def\S{\mathbb{S}}
\def\Ss{\mathbb{S}}
\def\Sn{\S^n}

\def\Rn{\R^n}

\newcommand{\hessp}{\nabla^2_{\textsc{\tiny BFGS}} \pCQP(Q_p)}
\newcommand{\pQKP}{p^*_{\textsc{\tiny QKP}}}

\newcommand{\pCQP}{p^*_{\textsc{\tiny CQP}}}

\newcommand{\param}{param^*_{\textsc{\tiny QKP}}}

\newcommand{\LCI}{\textbf{LCI}\,}
\newcommand{\LCIp}{\textbf{LCI}}
\newcommand{\SKILS}{\textbf{SKILS}\,}
\newcommand{\SKILSp}{\textbf{SKILS}}
\newcommand{\ECI}{\textbf{ECI}\,}
\newcommand{\ECIp}{\textbf{ECI}}
\newcommand{\CI}{\textbf{CI}\,}
\newcommand{\CIp}{\textbf{CI}}
\newcommand{\CILS}{\textbf{CILS}\,}
\newcommand{\CILSp}{\textbf{CILS}}
\newcommand{\SCILS}{\textbf{SCILS}\,}
\newcommand{\SCILSp}{\textbf{SCILS}}
\newcommand{\SCI}{\textbf{SCI}\,}
\newcommand{\SCIp}{\textbf{SCI}}

\newcommand{\KPol}{\textbf{KPol}\,}
\newcommand{\KPolp}{\textbf{KPol}}

\newcommand{\QKP}{\textbf{QKP}\,}
\newcommand{\QKPp}{\textbf{QKP}}
\newcommand{\QSDP}{\textbf{QSDP}\,}

\newcommand{\MILP}{\textbf{MILP}\,}

\newcommand{\MIQP}{\textbf{MIQP}\,}

\newcommand{\LSDP}{\textbf{LSDP}\,}

\newcommand{\cut}{\textbf{cut}}

\newcommand{\ConvRel}{\textbf{CRel}\,}
\newcommand{\ConvRelp}{\textbf{CRel}}

\newcommand{\LPR}{\textbf{LPR}\,}

\newcommand{\QPR}{\textbf{QPR}\,}

\newcommand{\CQP}{\textbf{CQP}\,}

\newcommand{\textdef}[1]{\textit{#1}\index{#1}}

\DeclareMathOperator{\conv}{{conv}}
\DeclareMathOperator{\trace}{{trace}}

\DeclareMathOperator{\svec}{{svec}}
\DeclareMathOperator{\sMat}{{sMat}}

\begin{document}

\begin{frontmatter}

\title{
	Parametric Convex Quadratic Relaxation	of the
	Quadratic Knapsack	Problem
}


%
%
%
%

\author[UFRJ,Corr]{M. Fampa}
\ead{fampa@cos.ufrj.br}
\author[UFRJ]{D. Lubke}
\ead{danielalubke@cos.ufrj.br}
\author[RIT]{F. Wang}
\ead{fewa@kth.se}
\author[UWat]{H. Wolkowicz}
\ead{hwolkowicz@uwaterloo.ca}

\address[UFRJ]{Universidade Federal do Rio de Janeiro, Brazil}
\address[RIT]{Dept. of Mathematics, Royal Institute of Technology, Sweden}
\address[UWat]{ Dept. of Combinatorics and Optimization, University of Waterloo, Waterloo, Canada}

\fntext[Corr]{Corresponding author: Marcia Fampa, PESC/COPPE, Universidade Federal do Rio de Janeiro, 
Caixa postal 68511, 
Rio de Janeiro, RJ, 21941-972,
Brazil. Email: fampa@cos.ufrj.br}

\begin{abstract}
We consider a parametric convex quadratic programming (CQP)
relaxation for the quadratic knapsack problem (QKP). This relaxation maintains partial quadratic information from the original QKP by perturbing the objective function to obtain a concave quadratic term.  The nonconcave part generated by the perturbation is then linearized by a standard approach that lifts the problem to matrix space. 
We present a primal-dual interior point method to optimize the perturbation of the quadratic function, in a search for the tightest upper bound for the QKP. We prove that the same perturbation approach, when applied in the context of semidefinite programming (SDP) relaxations  of the QKP, cannot improve the upper bound given by the corresponding  linear SDP relaxation. The result also applies to more general integer quadratic problems.  
Finally, we propose new valid inequalities on the lifted matrix variable, derived from cover and knapsack inequalities for the QKP, and present separation problems to generate cuts for the current solution of the CQP  relaxation.
Our best bounds are obtained alternating between
optimizing the parametric quadratic re\-lax\-a\-tion
over the perturbation and applying cutting planes generated by the valid inequalities proposed.
\end{abstract}


\end{frontmatter}

\section{Introduction}
We study  a convex quadratic programming (CQP) relaxation of the quadratic knapsack problem (QKP),
\hypertarget{modelqkp}{
\begin{equation}
\label{eq:qkp}
\text{(\QKP)} \qquad  \qquad
\begin{array}{rrl}
	\pQKP:= &\max & x^T Qx \\
	&\text{s.t.} & w^T x \leq c \\
	&& x\in \{0,1\}^n,
\end{array}
\index{optimal value \QKPp, $\pQKP$}
\index{$\pQKP$, optimal value \QKPp}
\index{$Q$, profit matrix}
\index{profit matrix, $Q$}
\index{$c$, knapsack capacity}
\index{knapsack capacity, $c$}
\index{$w$, item weights}
\index{item weights, $w$}
\end{equation}
}
where $Q\in \Sn$ is a symmetric $n\times n$ nonnegative integer  profit
matrix, $w\in \Z^n_{++}$ is a vector of positive integer weights for the items,
and $c\in \Z_{++}$  is the knapsack capacity with
$c\geq w_i,$ for all $i\in N:=\{1,  \ldots,n\}$. 
The binary (vector)  variable $x$ indicates which items are chosen for the
knapsack, and the inequality in the model, known as a  knapsack inequality, ensures that the selection of items does not exceed the knapsack capacity. We note that any linear costs in the objective can
be included on the diagonal of $Q$ by exploiting the $\{0,1\}$
constraints and, therefore,  are not considered.

The QKP was introduced in \cite{gallo} and was proved to be NP-Hard in the strong sense by reduction from the clique problem. The quadratic knapsack problem is a generalization of the knapsack problem, which has the same feasible set of the QKP, and a linear objective function in $x$. The linear knapsack problem can be solved in pseudo-polynomial time using dynamic programming approaches with complexity of $O(nc)$.

The QKP appears in a wide variety of fields, such as biology, logistics, capital budgeting, telecommunications and graph theory, 
and has received a lot of attention in the last decades. Several papers have proposed branch-and-bound algorithms for the QKP, and the main difference between them is the method used to obtain upper bounds for the subproblems
\cite{Chaillou89,Billionnet14,Caprara99,
HelRenWeis:96,MR1772826}. The well known trade-off between the strength of the bounds and the computational effort required to obtain them is intensively discussed in \cite{Pisinger2007}, where semidefinite programming (SDP) relaxations proposed in \cite{HelRenWeis:96} and \cite{MR1772826} are presented as the strongest relaxations for the QKP. The linear programming (LP) relaxation proposed in \cite{Billionnet14}, on the other side, is presented as the most computationally inexpensive. 

Both the SDP and the LP relaxations have a common feature, they are defined in the symmetric matrix lifted space determined by the equation $X=xx^T$, and by the replacement of the quadratic objective function in \hyperlink{modelqkp}{\QKPp} with a linear function in $X$, namely,  
$\trace (QX)$. As the constraint $X=xx^T$ is nonconvex, it is relaxed by convex constraints in the relaxations. The well known McCormick inequalities  \cite{McCormick-1976}, and also the semidefinite constraint, $X-xx^T\succeq 0$, 
have  been extensively used to relax the nonconvex constraint $X=xx^T$, in relaxations of the QKP. 

In this paper, we investigate a  CQP relaxation for the QKP, where instead of linearizing the objective function, we perturb the objective function Hessian $Q$,  and maintain the (concave) perturbed version of the quadratic function in the objective,  linearizing only the  remaining part derived from the perturbation. Our relaxation is a  parametric convex quadratic problem, defined as a function of a matrix parameter $Q_p$, such that $Q-Q_p\preceq 0$. This matrix parameter is iteratively optimized by a primal-dual interior point method (IPM) to generate the best possible bound for the QKP. During this iterative procedure, valid cuts are added to the formulation to strengthen the relaxation, and the search for the best perturbation is adapted accordingly. Our procedure  alternates between
optimizing the matrix parameter and applying cutting planes generated by valid inequalities. At each iteration of the procedure, a new bound for the QKP is computed, considering  the updated matrix parameter and the cuts already added to the relaxation.


A similar approach to handle nonconvex quadratic functions consists in decomposing it as a difference of convex (DC) quadratic function \cite{Horst1999}. DC decompositions have been extensively used in the literature to generate convex quadratic relaxations of nonconvex quadratic problems. See, for example, \cite{Fampa2017} and references therein. Unlike the approach used in DC decompositions, we do not necessarily decompose $x^T Qx$ as a difference of convex functions, or equivalently, as a sum of a convex and a concave function. Instead, 
we decompose it as a sum of a concave function and a quadratic term derived from the perturbation applied to $Q$. This perturbation can be any symmetric matrix $Q_p$, such that  $Q-Q_p\preceq 0$. 

In an attempt to obtain stronger bounds, we also investigated the  parametric convex quadratic SDP problem, where we add to our CQP relaxation, the positive  semidefinite constraint $X-xx^T\succeq 0$. An IPM  could also be applied to this parametric problem in order to generate the best possible bound. Nevertheless, we prove an interesting result concerning the relaxations, in case the constraint $X-xx^T\succeq 0$ is imposed: the tightest bound generated by the parametric quadratic SDP relaxation is obtained when the perturbation $Q_p$ is equal to $Q$, or equivalently, when we linearize the entire objective function, obtaining the standard linear SDP relaxation. We conclude, therefore, that keeping  the (concave) perturbed version of the quadratic function in the objective of the SDP relaxation does not lead to a tighter bound. 



Another contribution of this work is the development of valid inequalities for the CQP relaxation on the lifted matrix variable. The inequalities are first derived from cover inequalities for the knapsack problem. The idea is then extended to knapsack inequalities. Taking advantage of the lifting $X:=xx^T$, we propose new valid inequalities that can also be applied to more general relaxations of binary  quadratic programming problems that use the same lifting.  We discuss how cuts for the quadratic relaxation  can be obtained by the solution of separation problems, and investigate possible dominance relation between the inequalities proposed.

Finally, we present our  algorithm \verb;CWICS; (Convexification With Integrated Cut Strengthening), where we iteratively improve the upper bound for the QKP by optimizing the choice of the perturbation of the objective function and adding cutting planes to the relaxation. At each iteration, lower bounds for the problem are also generated from feasible solutions constructed from a rank-one approximation of the solution of the CQP relaxation. 

In Section \ref{sect:convexification}, we introduce our parametric convex quadratic relaxation for the QKP. In Section \ref{sec:OptParametric}, we explain how we optimize the parametric problem over the perturbation of the objective; i.e.,  we present the IPM applied to obtain the perturbation that leads to the best possible bound. In Section \ref{sec:sdprel}, we present our conclusion about the parametric quadratic SDP relaxation.
In Section \ref{sec:valid_ineq}, we introduce new valid inequalities on the lifted matrix variable of the convex quadratic model, and we describe how cutting planes are obtained by the solution of separation problems. In Section \ref{SectionFeaSolQAP}, we present the heuristic used to generate lower bounds to the QKP. 
In Section \ref{sect:numerics}, we present our algorithm \verb;CWICS;, and discuss our numerical experiments, and in Section \ref{sec:conclusion}, we present our final remarks.

\subsubsection*{Notation}If $A\in\mathbb{S}^n$, then $\svec(A)$ is a vector  whose entries come from $A$ by stacking up its `lower half', i.e., 
\[
\svec(A):=(a_{11},  \ldots,a_{n1},a_{22},  \ldots, a_{n2},  \ldots,a_{nn})^T\in \mathbb{R}^{n(n+1)/2}~.
\]

The operator $\sMat$ is the inverse of $\svec$, i.e., $\sMat(\svec(A))=A$. 

We also denote by $\lambda_{\min}(A)$, the smallest eigenvalue of $A$ and by $\lambda_i(A)$ the $i^{th}$ largest eigenvalue of $A$. 

To facilitate the reading of the paper, Table \ref{tab_acronyms} relates the acronyms used with the associated equations
numbers.

\begin{table}[ht]
\centering 
\begin{tabular}{cc|cc}
 \hline
\hyperlink{modelqkp}{\QKP}&\eqref{eq:qkp} &\hyperlink{ciineq}{\CI}&\eqref{ci}\\
\hyperlink{modelqkplifted}{\QKP$_{\mbox{lifted}}$} &\eqref{eq:qkplifted}&\hyperlink{eciineq}{\ECI}&\eqref{eci}\\
\hyperlink{modellinsdp}{\LPR}&\eqref{linsdp}&\hyperlink{lciineq}{\LCI}&\eqref{ineq_lci}\\
 \hyperlink{modelcqppert}{\CQP$_{Q_p}$}&\eqref{eq:qkppert}&\hyperlink{sciineq}{\SCI}&\eqref{cut1}\\
\hyperlink{modellsdp}{\LSDP}&\eqref{eq:lsdp} &\hyperlink{cilsineq}{\CILS}&\eqref{cut:gclique2}\\
\hyperlink{modelrelqkplifted}{\QSDP$_{Q_p}$}&\eqref{eq:relqkplifted} &\hyperlink{scilsineq}{\SCILS}&\eqref{cut2}\\
&&\hyperlink{skilsineq}{\SKILS}&\eqref{eq:skils}\\
\hline
\end{tabular}
\caption{Equations number corresponding to acronyms}
\label{tab_acronyms}
\end{table}
We also show the standard abbreviations used in the paper in Table \ref{tab:abb}.
\begin{table}[ht]
\centering 
\begin{tabular}{c|c}
 \hline
CQP&Convex Quadratic Programming\\
QKP&Quadratic Knapsack Problem\\
SDP&Semidefinte Programming\\
MIQP&Mixed Integer Quadratic Programming\\
MILP&Mixed Integer Linear Programming\\
\hline
\end{tabular}
\caption{List of abbreviations}
\label{tab:abb}
\end{table}


\section{A Parametric Convex Quadratic Relaxation}
\label{sect:convexification}
In order to construct a convex relaxation for \hyperlink{modelqkp}{\QKPp}, we start by considering the following  standard reformulation of the problem in the lifted space of symmetric matrices, defined by the lifting $X:=xx^T$.
\hypertarget{modelqkplifted}{
\begin{equation}
\label{eq:qkplifted}
\text{(\QKP$_{\mbox{lifted}}$)} \qquad  \qquad
\begin{array}{rrl}
    p^*_{\textsc{\tiny QKP$_{\mbox{lifted}}$}}:= &\max &  \trace (Q X)\\
	&\text{s.t.} & w^T x \leq c \\
&&X=xx^T\\
	&& x\in \{0,1\}^n.
\end{array}
\end{equation}
}

We consider an initial LP relaxation of \hyperlink{modelqkp}{\QKPp}, given by 
\hypertarget{modellinsdp}{
\begin{equation}
\label{linsdp}
(\LPR) \qquad  \qquad \begin{array}{lll}
   \max & \trace (QX)\cr
 \mbox{s.t.}  & (x,X) \in \mathcal{P},
\end{array}
\end{equation}
}
where $\mathcal{P}\subset [0,1]^n \times \mathbb{S}^n $ is a bounded polyhedron, such that 
\[\{ (x,X) \; : \;  w^T x \leq c, \; X=xx^T, \; x\in \{0,1\}^n\} \subset \mathcal{P}.\]

\subsection{The perturbation of the quadratic objective}

Next, we  propose a convex quadratic relaxation with the same feasible set as \hyperlink{modellinsdp}{\LPR} , but maintaining a concave perturbed version of the quadratic objective function of  \hyperlink{modelqkp}{\QKPp}, and linearizing only the remaining nonconcave part derived from the perturbation. 
More specifically, we choose $Q_p \in \Sn$ such that
\begin{equation}
\label{eq:QQnQp}
Q-Q_p \preceq 0,
\end{equation}
and we get
\begin{align*}
x^TQx&=  x^T(Q-Q_p)x +x^TQ_p x =
 x^T(Q-Q_p)x +\trace (Q_p xx^T)\\
  &=x^T(Q-Q_p)x +\trace (Q_p X).
\end{align*}
Finally, we define the parametric convex quadratic relaxation of \hyperlink{modelqkp}{\QKP}:
\hypertarget{modelcqppert}{
\begin{equation}
\label{eq:qkppert}
(\CQP_{Q_p}) \qquad  \qquad
\begin{array}{rll}
	\textdef{$\pCQP(Q_p)$} := \max & x^T (Q-Q_p)x +\trace (Q_pX)\\
\mbox{s.t.}  & (x,X) \in \mathcal{P}.
\end{array}
\end{equation}
}

\section{Optimizing the parametric problem over the parameter $Q_p$}
\label{sec:OptParametric}
The upper bound $\pCQP(Q_p)$ in the convex quadratic problem $\hyperlink{modelcqppert}{\CQP_{Q_p}}$
depends on the feasible perturbation $Q_p$ of the Hessian $Q$.
To find the best upper bound, we consider the parametric problem
\begin{equation}
	\label{eq:paramQQQ}
\param := \min_{Q-Q_p\preceq 0} \, \pCQP(Q_p).
\end{equation}
We solve \eqref{eq:paramQQQ} with a primal-dual interior-point method (IPM), and we describe in this section how the search direction of the algorithm is obtained at each iteration. 

We start
with minimizing a log-barrier function.  We use the
barrier function, $B_\mu (Q_p,Z)$
with barrier parameter, $\mu>0$, to obtain the barrier problem
\index{$\mu>0$, barrier parameter}
\index{$B_\mu (Q_p,Z)$, barrier function}
\begin{equation}
\label{eq:barprobl}
\begin{array}{rll}
\min  & B_\mu (Q_p,Z) := \pCQP(Q_p) -\mu \log \det Z
\\ \text{s.t.} &  Q-Q_p + Z = 0 & (:\Lambda)
\\             &  Z \succ 0,    &
\end{array}
\end{equation}
where $Z\in \Sn$ and $\Lambda\in \Sn$ denote, respectively,  the slack and the dual symmetric matrix variables. We consider the 
Lagrangian function
\[
L_\mu (Q_p,Z,\Lambda) := \pCQP(Q_p) -\mu \log \det Z
 + \trace(( Q-Q_p + Z)\Lambda).
\]

Some important points should be emphasized here. We first note that the objective function for $\pCQP(Q_p)$ is linear  in $Q_p$,
i.e.,~this function is the maximum of linear functions over feasible
points $x,X$. Therefore, this is a convex function.

Moreover, as will be detailed next, the search direction of the IPM, computed at each iteration of the algorithm, depends on the optimum solution  $x=x(Q_p),
X=X(Q_p)$ of $\hyperlink{modelcqppert}{\CQP_{Q_p}}$, for a fixed matrix $Q_p$. At each iteration of the IPM, we have $Z \succ 0$, and therefore $Q-Q_p\prec 0$. Thus, problem  $\hyperlink{modelcqppert}{\CQP_{Q_p}}$ maximizes a strictly concave quadratic function, subject to linear constraints over a compact set $\mathcal{P}$, and consequently, has a unique optimal solution (see e.g. \cite{doi:10.1002/wics.1344}).
From standard sensitivity analysis results,
e.g. ~\cite[Corollary 3.4.2]{Fia:83},\cite{Hog:73b}, \cite[Theorem 1]{doi:10.1137/0114053},
 as the optimal solution
 $x=x(Q_p),
X=X(Q_p)$ is unique, the function $\pCQP(Q_p)$ is differentiable and the gradient is obtained by differentiating the
Lagrangian function.

Since $Q_p$ appears only in the
objective function in $\hyperlink{modelcqppert}{\CQP_{Q_p}}$, and
\[
x^T (Q-Q_p)x +\trace (Q_pX)=
  x^TQx +   \trace (Q_p(X-xx^T)),
\]
we get a directional derivative at $Q_p$ in the direction
$\Delta Q_p$,
\[
D (\pCQP(Q_p);\Delta Q_p) =
        \max_{\text{optimal }x,X}\trace((X-xx^T )\Delta Q_p).
\]
Since we have a unique optimum $x=x(Q_p),
X=X(Q_p)$, we get the gradient
\begin{equation}
\label{grad}
\nabla \pCQP(Q_p) =X -xx^T.
\end{equation}
The gradient of the barrier function,  is then
\[
\nabla B_\mu (Q_p) =  (X-xx^T)  -\mu Z^{-1}.
\]

The optimality conditions for \eqref{eq:barprobl}
are obtained by differentiating the Lagrangian $L_\mu$
with respect to $Q_p,\Lambda,Z$, respectively,
\begin{equation}
\label{eq:logbaroptcond2}
\begin{array}{lrcll}
 \frac {\partial L_\mu}{\partial Q_p}: &
        \nabla \pCQP(Q_p)
             -\Lambda
                &=&0,
 \\ \frac {\partial L_\mu}{\partial \Lambda}: &
        Q-Q_p+Z
                &=&0,
 \\ \frac {\partial L_\mu}{\partial Z}: &
	-\mu Z^{-1} +\Lambda
                &=&0 , &
	(\text{or}) \,\, Z\Lambda-\mu I =0.
\end{array}
\end{equation}
This gives rise to the nonlinear  system
\begin{equation}
\label{eq:logbarnonlin}
\textdef{$G_\mu(Q_p,\Lambda,Z)$}=
\begin{pmatrix}
        \nabla \pCQP(Q_p) -\Lambda  \cr
        Q-Q_p+Z   \cr
	 Z\Lambda-\mu I
\end{pmatrix}=0, \qquad Z,\Lambda \succ 0.
\end{equation}

We use a BFGS approximation for the Hessian of $\pCQP$, since it is not guaranteed to be twice differentiable everywhere, and update it at each iteration (see \cite{Lewis2013}). We denote the 
approximation of \textdef{$\hessp$} by $B$, and begin with the approximation $B_0=I$. 
Recall that if
$Q_p^k,Q_p^{k+1}$ are two successive iterates with gradients
$\nabla \pCQP(Q^k_p),\nabla \pCQP(Q^{k+1}_p)$, respectively, with current
Hessian approximation $B_k\in \Ss^{n(n+1)/2}$, then we set
\[
Y_k:= \nabla \pCQP(Q^{k+1}_p)-\nabla \pCQP(Q^k_p), \qquad
S_k:= Q_p^{k+1}-Q_p^k,
\]
and,
\[
\upsilon:=\langle
Y_k,S_k\rangle, \qquad 
\omega:= \langle \svec(S_k), B_k \svec(S_k)\rangle.
\]
Finally, we update the Hessian approximation with
\[
B_{k+1}:= B_k+ \frac 1\upsilon \left(\svec(Y_{k})\svec(Y_{k}^T)\right) - \frac 1\omega\left(
       B_k \svec(S_{k})\svec(S_{k})^T B_k\right).
\]
We note that the curvature condition $\upsilon >0$ should be verified to guarantee the positive definiteness of the updated Hessian. In our implementation, we address this by skipping the BFGS update when $\upsilon$ is negative or too close to zero.

The  equation for the search direction is 
\begin{equation}
\label{eq:logbarnonlinJac}
\textdef{$G^\prime_\mu(Q_p,\Lambda,Z)$}
\begin{pmatrix}
        \Delta Q_p\cr
        \Delta \Lambda\cr
        \Delta Z
\end{pmatrix}
= -G_\mu(Q_p,\Lambda,Z),
\end{equation}
where 
\begin{equation}
\label{residuals}
G_\mu(Q_p,\Lambda,Z)=
\begin{pmatrix}
        \nabla \pCQP(Q_p) -\Lambda  \cr
        Q-Q_p+Z   \cr
	 Z\Lambda-\mu I
\end{pmatrix}=:
\begin{pmatrix}
        R_d\cr
        R_p\cr
        R_c
\end{pmatrix}.
\end{equation}
If $B$ is the current estimate of the Hessian, then the system becomes
\[
\left\{
\begin{array}{l}
\sMat( B \svec(\Delta Q_p) ) -\Delta \Lambda=-R_d, \\
- \Delta Q_p + \Delta Z=-R_p, \\
Z\Delta \Lambda + \Delta Z  \Lambda  = - R_c.
\end{array}\right.
\]
We can substitute for the  variables $\Delta \Lambda$ and $ \Delta Z$   in the third equation of the system. 
The elimination gives us a simplified system, and therefore, we apply it, using the following two equations for elimination and backsolving,
\begin{equation}
\label{eq:twoeqnelim}
\begin{array}{c}
\Delta \Lambda=\sMat( B \svec(\Delta Q_p) ) +R_d, \quad
\Delta Z=-R_p +\Delta Q_p.\\
\end{array}
\end{equation}
Accordingly, we have a single  equation to solve, and the system finally becomes
\[
Z\sMat( B \svec(\Delta Q_p) )
+(\Delta Q_p)\Lambda=-R_c-Z R_d+R_p\Lambda.
\]

We emphasize that  to compute the search direction at each iteration of our IPM, we need to update the residuals defined in \eqref{residuals}, and therefore we need the optimal solution 
$x=x(Q_p), X=X(Q_p)$ of the convex quadratic relaxation $\hyperlink{modelcqppert}{\CQP_{Q_p}}$ for the current perturbation $Q_p$. Problem $\hyperlink{modelcqppert}{\CQP_{Q_p}}$ is thus solved at each iteration of the IPM method, each time for a new perturbation $Q_p$, such that $Q-Q_p\prec 0$.

In Algorithm \ref{alg:ipm_iteration}, we present in details an iteration of the IPM. The algorithm is part of the complete framework used to generate bounds for $\hyperlink{modelqkp}{\QKPp}$, as described in Section \ref{sect:numerics}.

\begin{algorithm}
\small{
\caption{Updating the perturbation $Q_p$}
\label{alg:ipm_iteration}
\begin{quote}
Input:  $k$, $Q^{k}_p$, $Z^{k}$,  $\Lambda^{k}$, $x(Q^{k}_p)$, $X(Q^{k}_p)$, $\nabla \pCQP(Q^{k}_p)$, $B_{k}$, $\mu^k$, $\tau_{\alpha}:=0.95$, $\tau_{\mu}:=0.9$.\\
Compute the residuals:
\[
\begin{pmatrix}
        R_d\cr
        R_p\cr
        R_c
\end{pmatrix}:=
\begin{pmatrix}
        \nabla \pCQP(Q^{k}_p) -\Lambda^{k}   \cr
        Q-Q^{k}_p+Z^{k}   \cr
	 Z^{k}\Lambda^{k}-\mu^k I
\end{pmatrix}.
\]
\\
Solve the linear system for $\Delta Q_p$:
\[
Z^{k}\sMat( B_{k} \svec(\Delta Q_p) )
+(\Delta Q_p)\Lambda^{k}=-R_c-Z^{k}R_d+R_p\Lambda^{k}.
\]\\
Set:
\[
\Delta \Lambda:=\sMat( B_{k} \svec(\Delta Q_p) )+R_d,  \ \Delta Z:=-R_p +\Delta Q_p.
\]
\\
Update $Q_p$, $Z$ and $\Lambda$:
\[
Q^{k+1}_p := Q^{k}_p + \hat{\alpha}_p \Delta Q_p,\ Z^{k+1} := Z^{k}_p + \hat{\alpha}_p\Delta Z,\ \Lambda^{k+1} := \Lambda^{k} + \hat{\alpha}_d \Delta \Lambda,
\]
where 
\begin{align*}
&\textstyle \hat{\alpha}_p:=\tau_{\alpha}\times\min\{1,\mbox{argmax}_{\alpha_p}\{ Z^{k}_p + \alpha_p\Delta Z \succeq 0\}\},\\ 
&\textstyle \hat{\alpha}_d:=\tau_{\alpha}\times\min\{1,\mbox{argmax}_{\alpha_d}\{ \Lambda^{k} + \alpha_d \Delta \Lambda \succeq 0\}\}.
\end{align*}
\\
Obtain the optimal solution $x(Q^{k+1}_p)$, $X(Q^{k+1}_p)$ of relaxation $\hyperlink{modelcqppert}{\CQP_{Q_p}}$, where $Q_p:=Q^{k+1}_p$.
\\
Update the gradient of $\pCQP$:
\[
\nabla \pCQP(Q^{k+1}_p) := X(Q^{k+1}_p) - x(Q^{k+1}_p)x(Q^{k+1}_p)^T.
\]
\\
Update the Hessian approximation of $\pCQP$ (if $\upsilon>0$):
\begin{align*}
&Y_{k}:= \nabla \pCQP(Q^{k+1}_p)-\nabla \pCQP(Q^k_p), \ S_{k}:= Q_p^{k+1}-Q_p^k,\\
& \upsilon:=\langle
Y_{k},S_{k}\rangle, \ \omega:= \langle \svec(S_{k}), B_{k} \svec(S_{k})\rangle,\\
&B_{k+1}:= B_k+ \frac 1\upsilon \left(\svec(Y_{k})\svec(Y_{k}^T)\right) - \frac 1\omega\left(
       B_k \svec(S_{k})\svec(S_{k})^T B_k\right).
\end{align*}
\\
Update $\mu$:
\[
\mu^{k+1}:=  \tau_{\mu} \, \frac{\trace (Z^{k+1}\Lambda^{k+1})}{n}.
\]
\\
Output: $Q^{k+1}_p$, $Z^{k+1}$,  $\Lambda^{k+1}$, $x(Q^{k+1}_p)$, $X(Q^{k+1}_p)$, 
$\nabla \pCQP(Q^{k+1}_p)$, $B_{k+1}$, $\mu^{k+1}$.
\end{quote}
}
\end{algorithm} 

\remark{Algorithm is an interior-point method with a quasi-Newton step (BFGS). The object function we are minimizing is differentiable with exception possibly at the optimum. A complete convergence analysis of the algorithm is not in the scope of this paper, however, convergence analysis for some similar problems can be found in the literature. In \cite{ArmandGilbert2000}, it is shown that if the objective function is always differentiable and strongly convex, then it is globally convergent to the analytic center of the primal-dual optimal set when $\mu$ tends to zero and strict complementarity holds. }

\section{The parametric quadratic SDP relaxation}
\label{sec:sdprel}
In an attempt to obtain tighter bounds, a promising approach might seem to be to include the positive semidefinite constraint $X-xx^T\succeq 0$ in our parametric quadratic relaxation $\hyperlink{modelcqppert}{\CQP_{Q_p}}$, and solve a parametric convex quadratic SDP relaxation, also using an IPM.
Nevertheless, we show in this section that the convex quadratic SDP relaxation cannot generate a better bound than the linear SDP relaxation, obtained when we set $Q_p$ equal to $Q$. In fact, as shown below, the result applies not only to the QKP, but to more general problems as well. We emphasize here that the same result does not apply for  $\hyperlink{modelcqppert}{\CQP_{Q_p}}$. We could observe with our computational experiments that the best bounds were obtained by $\hyperlink{modelcqppert}{\CQP_{Q_p}}$, when we had $Q-Q_p\neq 0$, for all instances considered.

Consider the linear SDP problem given by  
\hypertarget{modellsdp}{
\begin{equation}
\label{eq:lsdp}
\text{(\LSDP)} \qquad  \qquad
\begin{array}{rrl}
	 p^*_{\textsc{\tiny LSDP}}:= &\sup &  \trace (Q X)\\
	&\text{s.t.} & (x,X)\in \mathcal{F} \\
&&X-xx^T\succeq 0,
\end{array}
\end{equation}
}
where $x\in \mathbb{R}^n$, $X\in \mathbb{S}^n$, and   $\mathcal{F}$ is any subset of $\mathbb{R}^n \times \mathbb{S}^n$. 

We now consider the parametric  SDP problem  given by
\hypertarget{modelrelqkplifted}{
\begin{equation}
\label{eq:relqkplifted}
\text{(\QSDP$_{Q_p}$)} \qquad  \qquad
\begin{array}{rrl}
   	 p^*_{\textsc{\tiny QSDP$_{Q_p}$}} := &\sup &   x^T (Q-Q_p) x +\trace (Q_p X)\\
	&\text{s.t.} & (x,X)\in \mathcal{F} \\
&&X-xx^T\succeq 0,
\end{array}
\end{equation}
} where
$Q-Q_p\preceq 0$.

\begin{theorem}\label{uselessthm}
Let $\mathcal{F}$ be any subset of $\mathbb{R}^n \times \mathbb{S}^n$. 
For any choice of matrix $Q_p$ satisfying $Q-Q_p\preceq 0$, we have \begin{equation}\label{comp}  p^*_{\textsc{\tiny QSDP$_{Q_p}$}} \geq  p^*_{\textsc{\tiny LSDP}}.\end{equation}  Moreover, $\inf\{  p^*_{\textsc{\tiny QSDP$_{Q_p}$}}  ~:~ Q-Q_p\preceq 0 \} = p^*_{\textsc{\tiny LSDP}}$.
\end{theorem}

\begin{proof}
Let    $(\tilde{x},\tilde{X})$ be a feasible solution for \hyperlink{modellsdp}{\LSDP}.
We have
 \begin{align}
   p^*_{\textsc{\tiny QSDP$_{Q_p}$}}  &\geq   \tilde{x}^T (Q-Q_p)\tilde{x} +\trace (Q_p \tilde{X})  \label{eq:optkkta1}\\
  & =  \trace((Q-Q_p)(\tilde{x}\tilde{x}^T-\tilde{X})) + \trace((Q-Q_p)\tilde{X}) \nonumber \\
  & + \trace (Q_p\tilde{X}) \label{eq:optkkta2} \\
  & =  \trace((Q-Q_p)(\tilde{x}\tilde{x}^T-\tilde{X})) + \trace(Q\tilde{X}) \label{eq:optkkta3} \\
  & \geq  \trace(Q\tilde{X}). \label{eq:optkkta4}
\end{align}
The inequality \eqref{eq:optkkta1} holds because  $(\tilde{x},\tilde{X})$ is also a feasible solution for \hyperlink{modelrelqkplifted}{\QSDP$_{Q_p}$}. The inequality in \eqref{eq:optkkta4} holds because of the negative semidefiniteness of $Q-Q_p$ and $\tilde{x}\tilde{x}^T-\tilde{X}$. Because $ p^*_{\textsc{\tiny QSDP$_{Q_p}$}}$ is an upper bound on the objective value of \hyperlink{modellsdp}{\LSDP} at any feasible solution, we can conclude that $p^*_{\textsc{\tiny QSDP$_{Q_p}$}} \geq  p^*_{\textsc{\tiny LSDP}}$. Clearly, 
$Q_p=Q$  satisfies $Q-Q_p=0\preceq 0$ and \hyperlink{modellsdp}{\LSDP}  is the same as \hyperlink{modelrelqkplifted}{\QSDP$_{Q_p}$} for this choice of $Q_p$. Therefore, 
 $\inf\{  p^*_{\textsc{\tiny QSDP$_{Q_p}$}}  ~:~ Q-Q_p\preceq 0 \} = p^*_{\textsc{\tiny LSDP}}$.
\end{proof}

Notice that in Theorem \ref{uselessthm} we do not require that $\mathcal{F}$   be convex nor  bounded. 
Also, in principle, for some choices of $Q_p$, 
we could have $p^*_{\textsc{\tiny QSDP$_{Q_p}$}}=+\infty$   with   
$ p^*_{\textsc{\tiny LSDP}}=+\infty$  or not. 

 \begin{remark}
As a corollary from Theorem  \ref{uselessthm}, we have  that  the upper bound for $\hyperlink{modelqkp}{\QKP}$, given by the solution of the quadratic relaxation $\hyperlink{modelcqppert}{\CQP_{Q_p}}$,  cannot be smaller than the upper bound given by the  solution of the SDP relaxation obtained from it, by adding the SDP constraint  $X-xx^T\succeq 0$ and setting $Q_p$ equal to $Q$.

%
%

\end{remark}

\section{Valid inequalities}
\label{sec:valid_ineq}
We are now interested in finding valid inequalities to strengthen relaxations of $\hyperlink{modelqkp}{\QKP}$ in the lifted space determined by the lifting $X:=xx^T$. Let us denote by \ConvRelp, any convex relaxation of $\hyperlink{modelqkp}{\QKP}$ in the lifted space, where the equation $X=xx^T$ was relaxed in some manner, by convex constraints, i.e., any convex relaxation of \hyperlink{modelqkplifted}{\QKP$_{\mbox{lifted}}$}  

We  note that if the  inequality
\begin{equation}\label{vineq}
\tau^Tx\leq \beta
\end{equation}
is valid for $\hyperlink{modelqkp}{\QKPp}$, where $\tau \in \Z_+^n$ and $\beta\in \Z_+$, then, as $x$ is nonnegative and $X:=xx^T$,

\begin{equation}\label{ineqY}
\left(x \; X\right) \begin{pmatrix} -\beta \cr
\tau
\end{pmatrix} \leq 0
\end{equation}
is a valid inequality for  \hyperlink{modelqkplifted}{\QKP$_{\mbox{lifted}}$}.  In this case, we say that \eqref{ineqY} is a valid inequality for
 \hyperlink{modelqkplifted}{\QKP$_{\mbox{lifted}}$} derived from the valid inequality \eqref{vineq} for $\hyperlink{modelqkp}{\QKPp}$.

\subsection{Preliminaries: knapsack polytope and cover inequalities}

We begin by recall the concepts of knapsack polytopes and cover inequalities.

The  knapsack polytope is the convex hull of the feasible points of the knapsack problem, 
\[KF :=  \{x \in \{0,1\}^n : w^T x \leq c \}.\]
\index{\KPolp, knapsack polytope}
\index{knapsack polytope, \KPolp}
\hypertarget{defkpol}{
\begin{definition}[zero-one knapsack polytope]
\[
\KPolp:= \conv(KF) = \conv( \{x \in \{0,1\}^n : w^T x \leq c \}).
\]
\end{definition}
}
\begin{proposition}
The dimension
\[
\dim (\hyperlink{defkpol}{\KPol}) = n,
\]
and \hyperlink{defkpol}{\KPol} is an independence system, i.e.,
\[
x \in \hyperlink{defkpol}{\KPolp}, y\in \{0,1\}^n, y\leq x \implies y\in \hyperlink{defkpol}{\KPolp}.
\]
\end{proposition}
\begin{proof}
Recall that $w_i\leq c, \forall i$. Therefore, all the unit vectors
$e_i\in\Rn$, as well as the zero vector,  are feasible, and the first statement follows. The second
statement is clear.
\end{proof}

Cover inequalities were originally presented in
\cite{Balas1975,Wolsey1975}; see also
\cite[Section II.2]{NW88}. These inequalities can be used in general
optimization problems with knapsack inequalities and  binary variables and, particularly,   in $\hyperlink{modelqkp}{\QKPp}$.

\hypertarget{ciineq}{\begin{definition}[cover inequality, \CIp]
The subset $C\subseteq N$ is a  cover if it
satisfies
\[
\sum_{j\in C} w_j > c.
\]
The (valid) \CI is
\begin{equation}
\label{ci}
\sum_{j\in C} x_j \leq |C|-1.
\end{equation}
The cover inequality is  minimal if no proper subset of $C$ is
also a cover.
\end{definition}}

\hypertarget{eciineq}{\begin{definition}[extended \CIp, \ECIp]
Let $w^*:=\max_{j\in C} w_j$ and define the extension of $C$ as
\[
E(C):= C\cup \{j \in N\backslash C : w_j \geq w^*\}.
\]
The  \ECI is
\begin{equation}
\label{eci}
\sum_{j\in E(C)} x_j \leq |C|-1.
\end{equation}
\end{definition}}

\hypertarget{lciineq}{\begin{definition}[lifted \CIp, \LCIp]
Given a cover $C$, let  $\alpha_j \geq 0, \forall j \in N\backslash C$, and $\alpha_j > 0$, for some $j \in N\backslash C$,  such that 
\begin{equation}\label{ineq_lci}
\sum_{j\in C} x_j +
\sum_{j\in N\backslash C} \alpha_j x_j \leq |C|-1,
\end{equation}
is a valid inequality  for \hyperlink{defkpol}{\KPol}.
Inequality  \eqref{ineq_lci} is a \LCI.
\end{definition}
}

Cover inequalities are extensively discussed in
\cite{Hammer1975,BalasZemel1978,Balas1975,Wolsey1975,NW88,Atamturk2005}.
Details about the computational complexity of \hyperlink{lciineq}{\LCI} is presented in
\cite{Zemel1989,Gu1998}.
Algorithm \ref{alg:findLCI} \cite[page \pageref{alg:findLCI}]{Wolsey1998}, shows how to derive a facet-defining
\hyperlink{lciineq}{\LCI} from a given minimal cover $C$.

\begin{algorithm}
\small{
\caption{Procedure to find  a Lifted Cover Inequality}
\label{alg:findLCI}
\begin{quote}
Sort the elements in ascending $w_i$ order $i \in N \setminus C$, defining $\{i_1,i_2,  \ldots,i_r\}$. \\
For: {t=1 {\bf to} r}
\begin{quote}
\begin{equation}
\label{KPlci}
\begin{array}{cll}
\zeta_t = &\max & \sum_{j=1}^{t-1} \alpha_{i_j} x_{i_j} + \sum_{i\in C} x_i\\
              &\mbox{s.t.}  & \sum_{j=1}^{t-1} w_{i_j} x_{i_j} + \sum_{i\in C} w_ix_i \leq c -w_{i_t} \\
                     && x \in \{0,1\}^{|C|+t-1}.
\end{array}
\end{equation}

Set $\alpha_{i_t} = |C|-1 -\zeta_t$.
\end{quote}
End
\end{quote}
}
\end{algorithm}

\subsection{Adding cuts to the relaxation}

Given a solution  $(\bar{x},\bar{X})$
of \ConvRelp, our initial goal is to obtain a valid inequality for \hyperlink{modelqkplifted}{\QKP$_{\mbox{lifted}}$} derived from a \hyperlink{ciineq}{\CI} that  is violated by $(\bar{x},\bar{X})$. A \hyperlink{ciineq}{\CI} is  formulated as 
$\alpha^Tx\leq e^T\alpha - 1$, where $\alpha\in \{0,1\}^n$ and $e$ denotes the
vector of ones. We then search for the \hyperlink{ciineq}{\CI} that maximizes the sum of the violations among the inequalities in $\bar{Y}\cut (\alpha)  \leq  0$, where  $\bar{Y}:=\left(\bar{x} \; \bar{X}\right)$ and
\[
\cut (\alpha)=\begin{pmatrix} -e^T\alpha +1 \cr
\alpha
\end{pmatrix}.
\]
To obtain such a \hyperlink{ciineq}{\CI},  we solve the following linear knapsack problem, 
\begin{equation} \label{eq:optextreme}
v^*:=\max_{\alpha}\{ e^T \bar{Y}\cut(\alpha) :
w^T\alpha \geq c+1, \,
 \alpha \in \{0,1\}^n\}.
\end{equation}

Let $\alpha^*$ solve \eqref{eq:optextreme}. If $v^*>0$, then at least one valid inequality in the following set of $n$ scaled cover inequalities, denoted in the following by \SCIp, is violated by $(\bar{x},\bar{X})$.

\hypertarget{sciineq}{
\begin{equation}
\label{cut1}
\left(x\; X\right)\begin{pmatrix} -e^T\alpha^* +1 \cr
\alpha^*
\end{pmatrix} \leq 0.
\end{equation}
}

Based on the following theorem, we note that to strengthen cut \eqref{cut1}, we may apply Algorithm \ref{alg:findLCI} to the \hyperlink{ciineq}{\CI} \;obtained, lifting it to an \hyperlink{lciineq}{\LCI}, and finally add the  valid inequality \eqref{ineqY} derived from the \hyperlink{lciineq}{\LCI} to \ConvRelp.

\begin{theorem}\label{rem_dom1}
The valid inequality \eqref{ineqY} for \hyperlink{modelqkplifted}{\QKP$_{\mbox{lifted}}$}, which is derived from a valid \hyperlink{lciineq}{\LCI}, dominates all inequalities derived from a \hyperlink{ciineq}{\CI} that can be lifted to the \hyperlink{lciineq}{\LCI}. 
\end{theorem}
\begin{proof}
Consider the \hyperlink{lciineq}{\LCI} \eqref{ineq_lci} derived from a \hyperlink{ciineq}{\CI} \eqref{ci} for $\hyperlink{modelqkp}{\QKPp}$. The corresponding scaled cover inequalities \eqref{ineqY} derived from the \hyperlink{ciineq}{\CI} and the \hyperlink{lciineq}{\LCI} are, respectively, 
\[
\sum_{j\in C} X_{ij} 
 \leq (|C|-1)x_i, \;\; \forall i\in N,
\]
and
\[
\sum_{j\in C} X_{ij} +
\sum_{j\in N\backslash C} \alpha_j X_{ij} \leq (|C|-1)x_i,  \;\; \forall i\in N,
\]
where $\alpha_j \geq 0, \forall j \in N\backslash C$. Clearly, as all $X_{ij}$ are nonnegative, the second inequality  dominates the first, for all $i\in N$.
\end{proof}

\subsection{New valid inequalities in the lifted space}

As  discussed, after finding any valid inequality in the form of \eqref{vineq} for $\hyperlink{modelqkp}{\QKPp}$, we may add  the constraint \eqref{ineqY}
to \ConvRel when aiming at better bounds. We observe now, that besides \eqref{ineqY} we can also generate other valid inequalities in the lifted space by taking advantage of the lifting $X := xx^T$, and also of the fact that $x$ is binary. In the following, we show how the idea can be applied to cover  inequalities.

Let
\begin{equation}\label{valid_ineq}
\sum_{j\in C} x_j \leq \beta,
\end{equation}
where  $C \subset N $ and $\beta<|C|$, be a valid inequality for \hyperlink{defkpol}{\KPol}.

Inequality \eqref{valid_ineq} can be either a cover inequality, \hyperlink{ciineq}{\CI},  an extended cover inequality, \hyperlink{eciineq}{\ECI},  or a particular lifted cover inequality, \hyperlink{lciineq}{\LCI}, where $\alpha_j \in \{0,1\}, \forall j \in N\backslash C$ in \eqref{ineq_lci}. Furthermore, given a general \hyperlink{lciineq}{\LCI}, where $\alpha_j \in \Z_+$, for all $j \in N\backslash C$, a valid inequality of type \eqref{valid_ineq} can be constructed by replacing each $\alpha_j$ with $\min\{\alpha_j,1\}$ in the \hyperlink{lciineq}{\LCI}.

\hypertarget{cilsineq}{
\begin{definition}[Cover inequality in the lifted space, \CILS]
Let  $C\subset N $ and $\beta<|C|$ as in inequality \eqref{valid_ineq}, and also consider here that $\beta>1$. We define 
\begin{equation}\label{cut:gclique2}
\sum_{i ,j \in C, i < j} X_{ij} \leq \binom{\beta}{2}.
\end{equation}
as the \CILS derived from \eqref{valid_ineq}.
\end{definition}}

\begin{theorem}
 If inequality \eqref{valid_ineq} is valid for $\hyperlink{modelqkp}{\QKPp}$, then the \hyperlink{cilsineq}{\CILS} \; \eqref{cut:gclique2} is a valid inequality  for \hyperlink{modelqkplifted}{\QKP$_{\mbox{lifted}}$}.
\end{theorem}
\begin{proof} Considering \eqref{valid_ineq},  we conclude that at most $\binom{\beta}{2}$ products of variables $x_ix_j$, where $i,j\in C$,  can be equal to 1. Therefore, as $X_{ij}:=x_ix_j$, the result follows.
\end{proof}

\begin{remark} When $\beta > 1$, inequality \eqref{valid_ineq} is well known as a  clique cut, widely used to model decision problems, and frequently used as a cut in branch-and-cut algorithms. In this case, using similar idea to what was used to construct the \hyperlink{cilsineq}{\CILSp}, we conclude that it possible to fix
\[
 X_{ij}  = 0, \mbox{ for all } i,j  \in C, i < j.
\]
\end{remark}

Given a solution  $(\bar{x},\bar{X})$   of \ConvRelp, the following MIQP problem is a separation problem, which searches for  a \hyperlink{cilsineq}{\CILS} violated by  $\bar{X}$. 
\hypertarget{miqp1}{
\[
\begin{array}{lll}
z^*:=\, &\max_{\alpha,\beta,K} \trace(\bar{X}K) - \beta(\beta-1),&\multicolumn{1}{r}{(\mbox{\MIQP}_1)}\\
\mbox{s.t.}&w'\alpha \geq c+1,\\
&\beta=e'\alpha -1,\\
&K(i,i)=0,&i=1,  \ldots,n,\\
&K(i,j)\leq \alpha_i, & i,j=1,  \ldots,n, \; i<j,\\
&K(i,j)\leq \alpha_j, & i,j=1,  \ldots,n, \; i<j,\\
&K(i,j)\geq 0, & i,j=1,  \ldots,n, \; i<j,\\
& K(i,j)\geq \alpha_i+\alpha_j-1, & i,j=1,  \ldots,n, \; i<j,\\
&\alpha \in\{0,1\}^n, \; \beta\in \mathbb{R}, \; K\in \mathbb{S}^n.\\
\end{array}
\]
}

If $\alpha^*,\beta^*,K^*$ solves \hyperlink{miqp1}{\MIQP$_1$}, with $z^*>0$, the \hyperlink{cilsineq}{\CILS} given by $\trace(K^*X)\leq \beta^*(\beta^* - 1)$ is violated by $\bar{X}$. The binary vector $\alpha^*$ defines the \hyperlink{ciineq}{\CI} from which the cut is derived. The \hyperlink{ciineq}{\CI} is specifically given by ${\alpha^*}^Tx\leq e^T{\alpha^*}-1$ and $\beta^*(\beta^*-1)$ determines the right-hand side of the \hyperlink{cilsineq}{\CILSp}. The inequality is multiplied by 2 because we consider the variable $K$ as a symmetric matrix, in order to simplify the presentation of the model.

\begin{theorem}\label{rem_dom2}
The valid inequality \hyperlink{cilsineq}{\CILS} for \hyperlink{modelqkplifted}{\QKP$_{\mbox{lifted}}$}, which is de\-ri\-ved from a valid \hyperlink{lciineq}{\LCI} in the form \eqref{valid_ineq}, dominates any \hyperlink{cilsineq}{\CILS} derived from a \hyperlink{ciineq}{\CI} that can be lifted to the \hyperlink{lciineq}{\LCIp}.
\end{theorem}
\begin{proof}
As $X$ is  nonnegative, it is straightforward to verify that if $X$ satisfies a \hyperlink{cilsineq}{\CILS} derived from a \hyperlink{lciineq}{\LCIp}, $X$ also satisfies any \hyperlink{cilsineq}{\CILS} derived from a \hyperlink{ciineq}{\CI} that can be lifted to the \hyperlink{lciineq}{\LCIp}.
\end{proof}

Any feasible solution of \hyperlink{miqp1}{\MIQP$_1$} such that $\trace(\bar{X}K) > \beta(\beta-1)$ generates a valid inequality for \hyperlink{modelqkplifted}{\QKP$_{\mbox{lifted}}$} that is violated by $\bar{X}$. Therefore, we do not need to solve \hyperlink{miqp1}{\MIQP$_1$} to optimality to generate a cut. Moreover, to generate distinct cuts, we can solve \hyperlink{miqp1}{\MIQP$_1$} several times (not necessarily to optimality), each time adding to it, the following ``no-good" cut to avoid the previously generated cuts:

\begin{equation}
\label{no-gooda}
\sum_{i\in N} \bar{\alpha}(i) (1-\alpha(i)) \geq  1,
\end{equation}
where $\bar{\alpha}$ is the value of the variable $\alpha$ in the solution of \hyperlink{miqp1}{\MIQP$_1$} when generating the previous cut. 

We note that, if  $\alpha^*,\beta^*,K^*$ solves \hyperlink{miqp1}{\MIQP$_1$}, then ${\alpha^*}'x\leq e'\alpha^* -1$ is a valid \hyperlink{ciineq}{\CI} for $\hyperlink{modelqkp}{\QKPp}$, however it may not be a minimal cover. Aiming at generating stronger valid cuts, based in Theorem \ref{rem_dom2},  we might add to the objective function of \hyperlink{miqp1}{\MIQP$_1$}, the term $-\delta e'\alpha$, for some weight $\delta>0$. The objective function would then favor minimal covers, which could  be lifted to a facet-defining  \hyperlink{lciineq}{\LCIp}, that would finally generate the \hyperlink{cilsineq}{\CILSp}. We should also emphasize that if the \hyperlink{cilsineq}{\CILS} derived from a \hyperlink{ciineq}{\CI} is violated by a given $\bar{X}$, then clearly, the \hyperlink{cilsineq}{\CILS} derived from the \hyperlink{lciineq}{\LCI} will also be violated by $\bar{X}$.

Now, we also note that, besides defining one cover inequality in the lifted space considering all possible pairs of indexes in $C$, we can also define a set of 
cover inequalities in the lifted space, considering for each inequality, a partition of the indexes in $C$ into subsets of cardinality 1 or 2. In this case, the right-hand side of the inequalities is never greater than $\beta/2$. The idea is made precise below.

\hypertarget{scilsineq}{
\begin{definition}[Set of cover inequalities in the lifted space, \SCILSp]
\label{defscils}
Let  $C \subset N $ and $\beta<|C|$ as in inequality \eqref{valid_ineq}.
Let 
\begin{enumerate}
\item $C_{s}:=\{(i_1,j_1),  \ldots,(i_{p},j_{p})\}$
be a partition of $C$, if $|C|$ is even.
\item $C_{s}:=\{(i_1,j_1),  \ldots,(i_{p},j_{p})\}$ be a  partition of $C\setminus \{i_0\} \,\, \mbox{for each} \,\,  i_0 \in C$, if $|C|$ is odd and $\beta$ is odd.
\item $C_{s}:=\{(i_0,i_0),(i_1,j_1),  \ldots,(i_p,j_p)\}$, where $\{(i_1,j_1),  \ldots,(i_{p},j_{p})\}$ is a  partition of $C\setminus \{i_0\} \,\, \mbox{for each} \,\,  i_0 \in C$, if $|C|$ is odd and $\beta$ is even.
\end{enumerate}
In all cases, $i_k<j_k$ for all $k=1,  \ldots,p$.\\
The inequalities in the \SCILS derived from \eqref{valid_ineq} are given by
\begin{equation}
\label{cut2}
\displaystyle \sum_{(i,j)\in C_{s}} X_{ij} \leq \left\lfloor \frac{\beta}{2} \right\rfloor, 
\end{equation}
for all partitions $C_{s}$ defined as above.
\end{definition}
}

\begin{theorem}
\label{thmscils}
 If inequality \eqref{valid_ineq} is valid for $\hyperlink{modelqkp}{\QKPp}$, then the inequalities in the \hyperlink{scilsineq}{\SCILS} \; \eqref{cut2} are valid for \hyperlink{modelqkplifted}{\QKP$_{\mbox{lifted}}$}.
\end{theorem}
\begin{proof} 
The proof of the validity of \hyperlink{scilsineq}{\SCILS} is based on the lifting relation $X_{ij}=x_ix_j$. We note that if the binary variable $x_i$ indicates whether or not the item $i$ is selected in the solution, the variable $X_{ij}$ indicates whether or not the pair of items  $i$ and $j$, are both selected in the solution. 
\begin{enumerate}
\item If $|C|$ is even,  $C_s$ is a partition of $C$ in exactly $|C|/2$ subsets with two elements each, and therefore, if  at most $\beta$ elements of $C$ can be selected in the solution, clearly at most $ \left\lfloor \frac{\beta}{2} \right\rfloor$ subsets of  $C_s$ can also be selected.
\item  If $|C|$ and $\beta$ are odd,   $C_s$ is a partition of $C\setminus \{i_0\} $ in exactly $|C-1|/2$ subsets with two elements each, where $i_0$ can be any element of $C$. In this case, if  at most $\beta$ elements of $C$ can be selected in the solution, clearly at most  $\frac{\beta-1}{2} \left(= \left\lfloor \frac{\beta}{2} \right\rfloor\right)$ subsets of  $C_s$ can also be selected.
\item  If $|C|$ is odd and $\beta$ is even,   $C_s$ is the union of $\{(i_0,i_0)\}$ with  a partition of $C\setminus \{i_0\} $ in exactly $|C-1|/2$ subsets with two elements each, where $i_0$ can be any element of $C$. In this case, if  at most $\beta$ elements of $C$ can be selected in the solution, clearly at most  $\frac{\beta}{2}\left( = \left\lfloor \frac{\beta}{2} \right\rfloor\right)$ subsets of  $C_s$ can also be selected. 
\end{enumerate}
\end{proof}

Given a solution  $(\bar{x},\bar{X})$  of \ConvRelp, we now present a mixed linear integer programming   (MILP) separation problem, which searches for  an inequality in \hyperlink{scilsineq}{\SCILS} that is most violated by $\bar{X}$.
Let $A\in \{0,1\}^{n\times \frac{n(n+1)}{2}}$. In the first $n$ columns of $A$ we have the $n\times n$ identity matrix. In  the remaining $n(n-1)/2$ columns of the matrix, there are exactly two elements equal to 1 in each column. All columns are distinct. For example, for $n=4$,
$$
A:=\left(\begin{array}{cccccccccc}
1&0&0&0&1&1&1&0&0&0\\
0&1&0&0&1&0&0&1&1&0\\
0&0&1&0&0&1&0&1&0&1\\
0&0&0&1&0&0&1&0&1&1\\
\end{array}\right).
$$
The columns of $A$ represent all the subsets of  items in $N$ with one or  two elements. Let
\hypertarget{milp2}{
\[
\begin{array}{lll}
z^*:=\, &\max_{\alpha,v,K,y} \trace(\bar{X}K) - 2v,&\multicolumn{1}{r}{(\mbox{\MILP}_2)}\\
\mbox{s.t.}&w'\alpha \geq c+1,\\
&K(i,i)=2y(i),&i=1,  \ldots,n,\\
&\sum_{i=1}^n y(i) \leq 1, \\
&K(i,j) =  \sum_{t=n+1}^{n(n+1)/2} A(i,t)A(j,t))y(t), & i,j=1,  \ldots,n,i<j,\\
&v \geq (e'\alpha-1)/2 - 0.5,\\
&v \leq (e'\alpha-1)/2,\\
&y(t)\leq 1-A(i,t) + \alpha(i),  & i=1,  \ldots,n,\; t=1,  \ldots,\frac{n(n+1)}{2},\\
&\alpha \leq Ay \leq \ \alpha +\left(\frac{n(n+1)}{2}\right)(1-\alpha),\\ 
&\alpha \in\{0,1\}^n, \; y\in\{0,1\}^{\frac{n(n+1)}{2}},\\
& v\in \mathbb{Z}, \; K\in \mathbb{S}^n.
\end{array}
\]
}

If $\alpha^*,v^*,K^*,y^*$ solves \hyperlink{milp2}{\MILP$_2$}, with $z^*>0$, then the particular inequality in \hyperlink{scilsineq}{\SCILS} given by 
\begin{equation}
\label{partscils}
\trace(K^*X)\leq 2v^*
\end{equation}
 is violated by $\bar{X}$. The binary vector $\alpha^*$ defines the \hyperlink{ciineq}{\CI} from which the cut is derived. As the \hyperlink{ciineq}{\CI} is given by $\alpha^*x\leq e'\alpha^*-1$, we can conclude that the cut generated either belongs to case (1) or (3) in Definition \ref{defscils}. This fact is considered in the formulation of \hyperlink{milp2}{\MILP$_2$}.  The vector $y^*$ defines a partition $C_{s}$ as presented in case (3), if $\sum_{i=1}^n y(i) =1$, and in case (1),  otherwise. We finally note that the number 2 in the right-hand side of \eqref{partscils} is due to the symmetry of the matrix $K^*$. 

We now may repeat the observations made for \hyperlink{miqp1}{\MIQP$_1$}.

Any feasible solution of \hyperlink{milp2}{\MILP$_2$} such that $\trace(\bar{X}K) > 2v$ generates a valid inequality for \ConvRel, which is violated by $\bar{X}$. Therefore, we do not need to solve \hyperlink{milp2}{\MILP$_2$} to optimality to generate a cut. Moreover, to generate distinct cuts, we can solve \hyperlink{milp2}{\MILP$_2$} several times (not necessarily to optimality), each time adding to it, the following suitable ``no-good" cut to avoid the previously generated cuts:

\begin{equation}
\label{no-goodb}
\sum_{i=1}^{\frac{n(n+1)}{2}} \bar{y}(i) (1-y(i)) \geq  1,
\end{equation}
where $\bar{y}$ is the value of the variable $y$ in the solution of \hyperlink{milp2}{\MILP$_2$}, when generating the previous cut. 

 The \hyperlink{ciineq}{\CI}\;  ${\alpha^*}'x\leq e'\alpha^* -1$  may not be a minimal cover. Aiming at generating stronger valid cuts,  we might add again to the objective function of \hyperlink{milp2}{\MILP$_2$}, the term $-\delta e'\alpha$, for some weight $\delta>0$. The objective function would then favor minimal covers, which could  be lifted to a facet-defining  \hyperlink{lciineq}{\LCIp}. In this case, however, after computing the \hyperlink{lciineq}{\LCIp}, we have to solve \hyperlink{milp2}{\MILP$_2$} again, with $\alpha$ fixed at values that represent the \hyperlink{lciineq}{\LCIp}, and $v$ fixed so that the right-hand side of the inequality is equal to the right-hand side of the \hyperlink{lciineq}{\LCIp}. All components of $y$ that were equal to 1 in the previous solution of \hyperlink{milp2}{\MILP$_2$} should also be fixed at 1. The new solution of \hyperlink{milp2}{\MILP$_2$} would indicate the other subsets of $N$ to be added to $C_{s}$. One last detail should be taken into account. If the cover $C$ corresponding to the \hyperlink{lciineq}{\LCI}, is such that $|C|$ is odd and the right-hand side of the \hyperlink{lciineq}{\LCI} is also odd, then the cut generated should belong to case (2) in Definition \ref{defscils}, and \hyperlink{milp2}{\MILP$_2$} should be modified accordingly. Specifically, the second and third constraints in \hyperlink{milp2}{\MILP$_2$}, should be modified respectively to 
$$
\begin{array}{ll}
K(i,i)=0,&i=1,  \ldots,n,\\
\sum_{i=1}^n y(i) = 1. \\
\end{array}
$$

\begin{remark}
Let $\gamma:=|C|$. Then, the number of inequalities in the \hyperlink{scilsineq}{\SCILS} is
$$
\frac{\gamma !}{2^{(\frac{\gamma}{2})}(\frac{\gamma}{2}!)},
$$
\\
if $\gamma$ is even, or
$$
\gamma \times \frac{(\gamma-1)!}{2^{(\frac{\gamma-1}{2})}(\frac{\gamma-1}{2}!)},
$$ 
if $\gamma$ is odd.
 \;
\end{remark}

Finally, we extend the ideas presented above to the more general case of knapsack inequalities. We note that the following discussion applies to a general \hyperlink{lciineq}{\LCI}, where  $\alpha_j \in \Z_+, \forall j \in N\backslash C$.

Let
\begin{equation}\label{knapsack_ineq}
\sum_{j\in N} \alpha_j x_j \leq \beta.
\end{equation}
be a valid knapsack inequality for \hyperlink{defkpol}{\KPol}, with $\alpha_j,\beta\in\Z_+, \beta \geq \alpha_j, \forall j\in N$.

\hypertarget{skilsineq}{
\begin{definition}[Set of knapsack inequalities in the lifted space, \SKILSp]\label{def:skils}
Let $\alpha_j$ be the coefficient  of $x_j$ in \eqref{knapsack_ineq}. Let $\{C_{1},  \ldots,C_{q}\}$ be the partition of $N$, such that $\alpha_u=\alpha_v$, if $u,v\in C_{k}$ for some $k$,  and $\alpha_u\neq \alpha_v$, otherwise.  The knapsack inequality \eqref{knapsack_ineq} can then be rewritten as 
\begin{equation}\label{knapsack_ineq2}
\sum_{k=1}^q \left(\tilde{\alpha}_{k} \sum_{j\in C_{k}} x_j \right)\leq \beta.
\end{equation}
Now, for $k=1,\dots,q$, let $C_{l_k}:=\{(i_{k_1},j_{k_1}),  \ldots,(i_{k_{p_k}},j_{k_{p_k}})\}$, where $i<j$ for all $(i,j)\in C_{l_k}$, and 
\begin{itemize}
\item $C_{l_k}$ is a partition of $C_{k}$, if $|C_{k}|$ is even. 
\item $C_{l_k}$ is a partition of $C_{k}\setminus \{i_{k_0}\}$, where $i_{k_0} \in C_{k}$, if $|C_{k}|$ is odd.
\end{itemize}
The inequalities in the \SKILS corresponding to \eqref{knapsack_ineq} are given by
\begin{equation}\label{eq:skils}
\displaystyle \sum_{k=1}^q \left(\tilde{\alpha}_{k} X_{i_{k_0}i_{k_0}} + 2\tilde{\alpha}_{k}\sum_{(i,j)\in C_{l_k}}  X_{ij} \right)\leq \beta , 
\end{equation}
for all partitions $C_{l_k}$, $k=1,  \ldots,q$, defined as above, and for all $i_{k_0} \in  C_{k}\setminus C_{l_k}$.
 (If $|C_{k}|$ is even, $C_{k}\setminus C_{l_k} = \emptyset$, and the term in the variable $X_{i_{k_0}i_{k_0}}$ does not exist.)
\end{definition}
}

\begin{remark}
Consider $\{C_{1},  \ldots,C_{q}\}$ as in Definition \ref{def:skils}.
For $k=1,  \ldots,q$, let $\gamma_k:=|C_k|$ and define
$$
NC_{l_k}:=\frac{\gamma_k !}{2^{(\frac{\gamma_k}{2})}(\frac{\gamma_k}{2}!)},
$$
if $\gamma_k$ is even, or
$$
NC_{l_k}:=\gamma_k \times \frac{(\gamma_k-1)!}{2^{(\frac{\gamma_k-1}{2})}(\frac{\gamma_k-1}{2}!)},
$$ 
if $\gamma_k$ is odd.

Then, the number of inequalities in \hyperlink{skilsineq}{\SKILS} is
$$
\prod_{k=1}^{q}NC_{l_k}.
$$
\end{remark}

\begin{remark} \label{rembeta2}
If $\tilde{\alpha}_{k}$ is even for every $k$, such that $\gamma_k$ is odd, then the right-hand side $\beta$ of inequality  \eqref{eq:skils} may be replaced with $2\times \left\lfloor \frac{\beta}{2} \right\rfloor$, which will strengthen the inequality in case $\beta$ is odd. 

Note that the case where $\gamma_k:=|C_{k}|$ is even for every $k$, is a particular case contemplated by this remark, where the the tightness of the inequality can also be applied. 
\end{remark}  

\begin{corollary}
\item If inequality \eqref{knapsack_ineq} is valid for  \hyperlink{modelqkp}{\QKP}, then the inequalities \eqref{eq:skils}, in the \hyperlink{skilsineq}{\SKILS} ,\;  are valid for \hyperlink{modelqkplifted}{\QKP$_{\mbox{lifted}}$}, whether or not the modification suggested  in Remark \ref{rembeta2} is applied.
\end{corollary}
\begin{proof} The result is again verified, by using the same argument used in the proof of Theorem \ref{thmscils}, i.e.,  considering that $X_{ij}=1$, iff $x_i=x_j=1$.
\end{proof}

\subsection{Dominance relation among the new valid inequalities}

We start this subsection investigating whether \hyperlink{scilsineq}{\SCILS} dominates \hyperlink{cilsineq}{\CILS} or vice versa.

\begin{theorem}Let $C$ be the cover in \eqref{valid_ineq} and consider $\gamma := |C|$ to be even.
\begin{enumerate}
\item If  $\beta = \gamma - 1$, then the sum of all inequalities in \hyperlink{scilsineq}{\SCILS} is equivalent to \hyperlink{cilsineq}{\CILS}. Therefore, in this case, the set of inequalities in \hyperlink{scilsineq}{\SCILS} dominates \hyperlink{cilsineq}{\CILS}.
\item If $\beta < \gamma - 1$, there is no dominance relation between \hyperlink{scilsineq}{\SCILS} and \hyperlink{cilsineq}{\CILS}.
\end{enumerate}
\end{theorem}
\begin{proof}
Let $sum(\hyperlink{scilsineq}{\SCILS})$ denote the inequality obtained by adding all inequalities in \hyperlink{scilsineq}{\SCILS}, and let $rhs(sum(\hyperlink{scilsineq}{\SCILS}))$ denote its right-hand side (rhs). We have that 
$rhs(sum(\hyperlink{scilsineq}{\SCILS}))$ is equal to the number of inequalities in \hyperlink{scilsineq}{\SCILS} multiplied by the rhs of each inequality, i.e.:  
$$
rhs(sum(\hyperlink{scilsineq}{\SCILS})) = \frac{\gamma !}{2^{(\frac{\gamma}{2})}(\frac{\gamma}{2}!)} \times \left\lfloor \frac{\beta}{2} \right\rfloor.
$$
The coefficient of each variable $X_{ij}$ in $sum(\hyperlink{scilsineq}{\SCILS})$ ($coef_{ij}$) is given by the number of inequalities in the set \hyperlink{scilsineq}{\SCILS} in which $X_{ij}$ appears, i.e.:
$$
coef_{ij} = \frac{(\gamma-2) !}{2^{(\frac{(\gamma-2)}{2})}(\frac{(\gamma-2)}{2}!)}
$$
Dividing $rhs(sum(\hyperlink{scilsineq}{\SCILS}))$ by $coef_{ij}$, we obtain  
\begin{equation}\label{rhs1}
rhs(sum(\hyperlink{scilsineq}{\SCILS})) / coef_{ij} = (\gamma-1) \times \left\lfloor \frac{\beta}{2} \right\rfloor.
\end{equation}
On the other side, the rhs of \hyperlink{cilsineq}{\CILS} is:
\begin{equation}\label{rhs2}
rhs(\hyperlink{cilsineq}{\CILS}) = \binom{\beta}{2} = \frac{\beta(\beta-1)}{2}.
\end{equation}

\begin{enumerate}
\item Replacing  $\beta$ with $\gamma - 1$, and $\left\lfloor \frac{\beta}{2} \right\rfloor$ with $\frac{\beta-1}{2}$ (since $\beta$ is odd), we obtain the result. 
\item Consider, for example, $C=\{1,2,3,4,5,6\}$ and $\beta=3$ ($\beta < \gamma -1$ and odd). In this case, the \hyperlink{cilsineq}{\CILS} becomes:
\begin{align*}
 X_{12}+X_{13}&+X_{14}+X_{15}+X_{16}+X_{23}+X_{24}  \\
&+X_{25}+X_{26}+X_{34}+X_{35}+X_{36}+X_{45}+X_{46}+X_{56}\leq 3.
\end{align*}
And a particular inequality in \hyperlink{scilsineq}{\SCILS} is
\begin{equation}\label{ex1}
X_{12}+X_{34}+X_{56}\leq 1.
\end{equation}
The solution $X_{1j}=1$, for $j=2,  \ldots, 6$, and all other variables equal to zero, satisfies all inequalities in \hyperlink{scilsineq}{\SCILS}, because only one of the positive variables appears in each inequality in the set. However, the solution does not satisfy \hyperlink{cilsineq}{\CILS}. On the other side, the solution $X_{12}=X_{34}=X_{56}=1$, and all other variables equal to zero, satisfies \hyperlink{cilsineq}{\CILS}, but does not satisfy \eqref{ex1}.

Now, consider 
$C=\{1,2,3,4,5,6\}$ and $\beta=4$ ($\beta < \gamma -1$ and even). In this case, the \hyperlink{cilsineq}{\CILS} becomes:
\begin{align*} X_{12}+X_{13}&+X_{14}+X_{15}+X_{16}+X_{23}+X_{24}    \\
&+X_{25}+X_{26}+X_{34}+X_{35}+X_{36}+X_{45}+X_{46}+X_{56}\leq 6.
\end{align*}
And a particular inequality in \hyperlink{scilsineq}{\SCILS} is
\begin{equation}\label{ex2}
X_{12}+X_{34}+X_{56}\leq 2.
\end{equation}
The solution $X_{1j}=1$, for $j=2,  \ldots, 6$, $X_{2j}=1$, for $j=3,  \ldots, 6$, and all other variables equal to zero, satisfies all inequalities in \hyperlink{scilsineq}{\SCILS}, because at most two of the positive variables appear in each inequality in the set. However, the solution does not satisfy \hyperlink{cilsineq}{\CILS}.
 On the other side, the solution $X_{12}=X_{34}=X_{56}=1$, and all other variables equal to zero, satisfies \hyperlink{cilsineq}{\CILS}, but does not satisfy \eqref{ex2}. 
\end{enumerate}
\end{proof}

\begin{theorem}Let $C$ be the cover in \eqref{valid_ineq} and consider $\gamma := |C|$ to be odd.
 Then there is no dominance relation between \hyperlink{scilsineq}{\SCILS} and \hyperlink{cilsineq}{\CILS}.

\end{theorem}
\begin{proof}

Consider, for example, $C=\{1,2,3,4,5\}$ and $\beta=3$ ($\beta$  odd). In this case, the \hyperlink{cilsineq}{\CILS} becomes:
$$ X_{12}+X_{13}+X_{14}+X_{15}+X_{23}+X_{24}+X_{25}+X_{34}+X_{35}+X_{45}\leq 3.
$$
And a particular inequality in \hyperlink{scilsineq}{\SCILS} is
\begin{equation}\label{ex3}
X_{23}+X_{45}\leq 1.
\end{equation}
The solution $X_{1j}=1$, for $j=1,  \ldots, 5$, and all other variables equal to zero, satisfies all inequalities in \hyperlink{scilsineq}{\SCILS}, because only one of the positive variables appears in each inequality in the set. However, the solution does not satisfy \hyperlink{cilsineq}{\CILS}. On the other side, the solution $X_{23}=X_{45}=1$, and all other variables equal to zero, satisfies \hyperlink{cilsineq}{\CILS}, but does not satisfy \eqref{ex3}.

Now, consider 
$C=\{1,2,3,4,5\}$ and $\beta=4$ ($\beta$ even). In this case, the \hyperlink{cilsineq}{\CILS} becomes:
$$ X_{12}+X_{13}+X_{14}+X_{15}+X_{23}+X_{24}+X_{25}+X_{34}+X_{35}+X_{45}\leq 6.
$$
And a particular inequality in \hyperlink{scilsineq}{\SCILS} is
\begin{equation}\label{ex4}
X_{11}+X_{23}+X_{45}\leq 2.
\end{equation}
The solution $X_{1j}=1$, for $j=1,  \ldots, 5$, $X_{2j}=1$, for $j=2,  \ldots, 5$, and all other variables equal to zero, satisfies all inequalities in \hyperlink{scilsineq}{\SCILS}, because at most two of the positive variables appear in each inequality in the set. However, the solution does not satisfy \hyperlink{cilsineq}{\CILS}.
 On the other side, the solution $X_{11}=X_{23}=X_{45}=1$, and all other variables equal to zero, satisfies \hyperlink{cilsineq}{\CILS}, but does not satisfy \eqref{ex4}.
\end{proof}

Now, we investigate if \hyperlink{scilsineq}{\SCILS} is just a particular case of \hyperlink{skilsineq}{\SKILS}, when  $\alpha_j\in\{0,1\}$, for all $j\in N$ in \eqref{knapsack_ineq}. 

\begin{theorem} In case the modification suggested in Remark \ref{rembeta2} is applied, then if $|C|$ is even in \eqref{valid_ineq}, \hyperlink{scilsineq}{\SCILS} becomes just a particular case of \hyperlink{skilsineq}{\SKILS}. In case $|C|$ is odd, however, the inequalities in \hyperlink{scilsineq}{\SCILS} are stronger. 
\end{theorem}
\begin{proof}
If $|C|$ is even, the result is easily verified. If $|C|$ is odd, the inequalities in \hyperlink{scilsineq}{\SCILS} become 
\[
 2\sum_{(i,j)\in C_{s}} X_{ij} \leq  \beta-1, 
\]
if  $\beta$ is odd, and
\[
\displaystyle 2X_{i_{0}i_{0}} + 2\sum_{(i,j)\in C_{s}} X_{ij} \leq  \beta, 
\]
if  $\beta$ is even, 
and the inequalities in \hyperlink{skilsineq}{\SKILS} become 
$$
X_{i_{0}i_{0}} + 2\sum_{(i,j)\in C_{s}}  X_{ij} \leq \beta, 
$$
for all $\beta$. In all cases,  $C_{s}$ is a partition of $C\setminus \{i_{0}\}$, where $i_{0} \in C$.

Either with $\beta$ even or odd, it becomes clear that \hyperlink{scilsineq}{\SCILS} is stronger than \hyperlink{skilsineq}{\SKILS}.
\end{proof}

\section{Lower bounds from solutions of the relaxations for  \QKP$_{\mbox{lifted}}$}
\label{SectionFeaSolQAP}
In order to evaluate the quality of the upper bounds obtained with \ConvRelp, we compare them with lower bounds for $\hyperlink{modelqkp}{\QKPp}$, given by feasible solutions constructed by a heuristic. We assume in this section that all variables in \ConvRel are constrained to the interval $[0,1]$.

Let $(\bar{x},\bar{X})$ be a solution of \ConvRelp. We initially apply principal component analysis (PCA) \cite{Jolliffe2010} to construct an  approximation to the solution of $\hyperlink{modelqkp}{\QKP}$ and then apply a special rounding procedure to obtain a feasible solution from it.
PCA selects the largest eigenvalue and the corresponding eigenvector of $\bar{X}$, denoted by
$\bar{\lambda}$ and $\bar{v}$, respectively. Then $\bar{\lambda} \bar{v} \bar{v}^T$ is a
rank-one approximation of $\bar{X}$. We set $\bar x = \bar{\lambda}^{\frac 12} \bar{v}$ to be an
approximation of the solution $x$ of $\hyperlink{modelqkp}{\QKPp}$. We note that $\bar{\lambda}> 0$ because $\bar{X}_{ii}>0$ for at least one index $i$ in the optimal solutions of the relaxations, and therefore, $\bar{X}$ is not negative semidefinite.  Finally, we round $\bar x$ to a binary solution that satisfies the knapsack capacity constraint, using the simple approach described in Algorithm \ref{alg:rounding}.
\begin{algorithm}
\small{
\caption{ A heuristic for the QKP }\label{alg:rounding}
\begin{quote}
Input: {the solution $\bar X$ from \ConvRel, the weight vector $w$, the capacity $c$}. \\
Let $\bar{\lambda}$ and $\bar{v}$ be, respectively, the largest eigenvalue and the corresponding eigenvector of $\bar{X}$.\\
Set $\bar x = \bar{\lambda}^{\frac 12} \bar{v}$. \\
Round $\bar x$ to $\hat x\in \{0,1\}^n$. \\
While{ $w^T \hat x > c$}
\begin{quote}
Set $i = \mbox{argmin}_{j \in N}\{\bar x_j | \bar x_j > 0\}$. \\
Set $\bar x_i = 0$, $\hat x_i = 0$.\\
\end{quote}
End\\
Output: a feasible solution $\hat x$ of $\hyperlink{modelqkp}{\QKPp}$.
\end{quote}
}
\end{algorithm}

\section{Numerical Experiments}
\label{sect:numerics} 
We summarize our algorithm \verb;CWICS; (Convexification With Integrated Cut Strengthening)   in Algorithm \ref{algframework}, where at each  iteration we update the perturbation $Q_p$ of the parametric relaxation and, at every $m$ iterations, we add to the relaxation, the valid inequalities considered in this paper, namely, 
 \hyperlink{sciineq}{\SCIp}, defined in \eqref{cut1}, \hyperlink{cilsineq}{\CILSp}, defined in \eqref{cut:gclique2}, and \hyperlink{scilsineq}{\SCILSp}, defined in \eqref{cut2}. 

\begin{algorithm}
\small{
\nonumber
\label{algframework}\caption{CWICS(ConvexificationWithIntegratedCutStrengthening)}
\begin{quote}
Input:  $Q\in \mathbb{S}^n$, $max.n_{cuts}$.\\
$k:=0$, $B_{0}:=I$, $\mu^{0}:=1$.\\
Let $\lambda_i(Q),v_i$ be the $i^{th}$ largest eigenvalue of $Q$ and corresponding eigenvector.\\
$Q_n:=\sum_{i=1}^{n}(-|\lambda_i(Q)|-1)v_iv_i'$ (or  $Q_n:=\sum_{i=1}^{n}(\min\{\lambda_i,-10^{-6}\})v_iv_i'$), $Q^{0}_p:= Q - Q_n$.\\
Solve  $\hyperlink{modelcqppert}{\CQP_{Q_p}}$, with $Q_p:=Q^{0}_p$, and obtain $x(Q^{0}_p)$, $X(Q^{0}_p)$.\\
$\nabla \pCQP(Q^{0}_p) := X(Q^{0}_p) - x(Q^{0}_p)x(Q^{0}_p)^T$.\\
$Z^{0}:=Q^{0}_p-Q$.\\
$\Lambda^{0}:= \nabla \pCQP(Q^{0}_p) +(2|\lambda_{\min}(\nabla \pCQP(Q^{0}_p)|+0.1)I$.\\
While (stopping criterium is violated) 
\begin{quote}
Run Algorithm \ref{alg:ipm_iteration}, where $Q^{k+1}_p$ is obtained and relaxation $\hyperlink{modelcqppert}{\CQP_{Q_p}}$,  with $Q_p:=Q^{k+1}_p$ is solved. Let $(x(Q^{k+1}_p), X(Q^{k+1}_p))$ be its optimal solution.\\
$upper.bound^{k+1} := \pCQP(Q^{k+1}_p)$.\\
Run Algorithm \ref{alg:rounding}, where  $\hat{x}$ is obtained.\\
$lower.bound^{k+1} := \hat{x}^TQ\hat{x}$. \\
If $k$ mod $m$ == 0\\
\begin{quote}
Solve problem \eqref{eq:optextreme} and obtain cuts \hyperlink{sciineq}{\SCI} in \eqref{cut1}. \\
Add  the  $\max\{n,max.n_{cuts}\}$ cuts \hyperlink{sciineq}{\SCI} with the largest violations at $(x(Q^{k+1}_p), X(Q^{k+1}_p))$, to $\hyperlink{modelcqppert}{\CQP_{Q_p}}$.\\
$n_{cuts}:=0$.\\
While ($n_{cuts}< max.n_{cuts}$ \&  \hyperlink{miqp1}{\MIQP$_1$} feasible)\\
\begin{quote}
Solve  \hyperlink{miqp1}{\MIQP$_1$} and add the \hyperlink{ciineq}{\CI} and \hyperlink{cilsineq}{\CILS} obtained   to $\hyperlink{modelcqppert}{\CQP_{Q_p}}$.\\
Add the ``no-good" cut \eqref{no-gooda} to  \hyperlink{miqp1}{\MIQP$_1$}.\\
$ n_{cuts}:= n_{cuts}+1$.\\
\end{quote}
End \\
$n_{cuts}:=0$.\\
While ($n_{cuts}< max.n_{cuts}$ \& \hyperlink{milp2}{\MILP$_2$} feasible)\\
\begin{quote}
Solve  \hyperlink{milp2}{\MILP$_2$} and add the \hyperlink{ciineq}{\CI} and \hyperlink{scilsineq}{\SCILS} obtained  to $\hyperlink{modelcqppert}{\CQP_{Q_p}}$.\\
Add the ``no-good" cut \eqref{no-goodb} to  \hyperlink{milp2}{\MILP$_2$}.\\
$ n_{cuts}:= n_{cuts}+1$.\\
\end{quote}
End \\
\end{quote}
End\\
$k:=k+1$.\\
\end{quote}
End\\
Output: Upper bound $upper.bound^{k}$, lower bound  $lower.bound^{k}$, and feasible solution $\hat{x}$ to $\hyperlink{modelqkp}{\QKPp}$. 
\end{quote}
}
\end{algorithm} 
The numerical  experiments performed had the following main purposes, \begin{itemize}
\item verify the impact of the  valid inequalities, \hyperlink{sciineq}{\SCIp}, \hyperlink{cilsineq}{\CILSp}, and \hyperlink{scilsineq}{\SCILSp}, when iteratively added to cut the current solution of a  relaxation of $\hyperlink{modelqkp}{\QKPp}$,
\item verify the effectiveness of the IPM described in Section \ref{sec:OptParametric} in decreasing the upper bound while optimizing the perturbation $Q_p$,
\item compute  the upper and lower bounds obtained with the proposed algorithmic approach described in \verb;CWICS; (Algorithm \ref{algframework}), and compare them, with the optimal solutions of the instances.
\end{itemize}

We coded \verb;CWICS;  in MATLAB,  version R2016b,  and ran the code on a notebook with an Intel Core i5-4200U CPU 2.30GHz, 6GB RAM,  running under  Windows 10. We used the primal-dual IPM implemented in Mosek, version 8, to solve the relaxation $\hyperlink{modelcqppert}{\CQP_{Q_p}}$, and, to solve the separation problems \hyperlink{miqp1}{\MIQP$_1$} and \hyperlink{milp2}{\MILP$_2$}, we use Gurobi, version 8. 

The input data used in the first iteration of  the IPM described in Algorithm  \ref{alg:ipm_iteration}  ($k=0$) are: $B_0=I$, $\mu^{0}=1$. 
We start with a matrix $Q^{0}_p$, such that $Q-Q^{0}_p$ is negative definite.
By solving $\hyperlink{modelcqppert}{\CQP_{Q_p}}$, with $Q_p:=Q^{0}_p$,  we obtain $x(Q^{0}_p)$, $X(Q^{0}_p)$, as its  optimal solution, and set $\nabla \pCQP(Q^{0}_p) := X(Q^{0}_p) - x(Q^{0}_p)x(Q^{0}_p)^T$. Finally, the positive definiteness of $Z^{0}$ and $\Lambda^{0}$ are assured by setting: 
$Z^{0}:=Q^{0}_p-Q$ and 
$\Lambda^{0}:= \nabla \pCQP(Q^{0}_p) +(2|\lambda_{\min}(\nabla \pCQP(Q^{0}_p)|+0.1)I$.

Our randomly generated test instances  were also used  by J. O. Cunha in \cite{Cunha2016}, who provided us with the instances data and with their optimal solutions.   
Each weight $w_j$, for $j\in N$,  was randomly selected in the interval $[1, 50]$, and the capacity  $c$, of the knapsack, was randomly selected in $[50,\sum_{j=1}^n w_j]$. The procedure used by Cunha to generate the instances was based on previous works \cite{Billionnet14,Caprara99,Chaillou89,gallo,Michelon96}.

The following labels identify the results presented in Tables \ref{tableCuts10_100} and \ref{table50_100f}--\ref{table100_100f}.
\begin{itemize}
\item   OptGap (\%):= ((upper bound - opt)/opt) $\times$ 100,   where opt is the optimal solution value (the relative optimality gap),
\item Time (sec) (the computational time to compute the bound),  
\item DuGap (\%) := (upper bound - lower bound)/(lower bound) $\times$ 100,   where the lower bound is computed as described in Sect. \ref{SectionFeaSolQAP} (the relative duality gap),
\item Iter (the number of iterations), 
\item Cuts (the number of cuts added to the relaxation),
\item Time$_{\tiny{\mbox{MIP}}}$ (sec) (the computational time to obtain cuts \hyperlink{cilsineq}{\CILS} and \hyperlink{scilsineq}{\SCILS}).
\end{itemize}

To get some insight into the effectiveness of the cuts proposed, we initially applied them to 10 small instances with $n=10$.  In Table \ref{tableCuts10_100}    
we present average results for this preliminary experiment, where we iteratively add the cuts 
 to the following  linear relaxation

\begin{equation}
\label{linsdpc}
({\tilde{\hyperlink{modellinsdp}{\LPR} }}) \, \begin{array}{lll}
   \max & \trace (QX)\cr
 \mbox{s.t.}  & \sum_{j=1}^n w_j x_j \leq c ,\\
&0\leq X_{ij}\leq 1, &\forall i,j\in N\\
&0\leq x_{i}\leq 1, &\forall i\in N\\
&X\in \mathbb{S}^n.
\end{array}
\end{equation}

In the first row of Table \ref{tableCuts10_100},  the results correspond to the solution of the linear relaxation $\tilde{\hyperlink{modellinsdp}{\LPR} }$ with no cuts. In \SCIp$_1$, we  add only the most violated cut from the $n$ cuts in  \hyperlink{sciineq}{\SCI}  to $\tilde{\hyperlink{modellinsdp}{\LPR} }$ at each iteration, and in the \hyperlink{sciineq}{\SCI} we  add all $n$ cuts. 
In \hyperlink{cilsineq}{\CILS} and \hyperlink{scilsineq}{\SCILSp}, we solve  MIQP and MILP   problems to find the most violated cut of each type. The last row of the table  (All) corresponds to results obtained when we add all $n$ cuts in  \hyperlink{sciineq}{\SCIp}, and one cut of each type, \hyperlink{cilsineq}{\CILS} \ and \hyperlink{scilsineq}{\SCILSp}. In these initial tests, we run up to 50 iterations, and in most cases,  stop  the algorithm when no more cuts are found to be added to the relaxation.

\begin{table}
\centering
\begin{tabular}{l|r|r|r|r|r}
\hline
	  \multicolumn{1}{l|}{\textbf{Method}}  & \multicolumn{1}{|c|}{\textbf{OptGap} }& \multicolumn{1}{|c|}{\textbf{Time}} & \multicolumn{1}{|c|}{\textbf{Iter} }& \multicolumn{1}{|c|}{\textbf{Cuts}} & \multicolumn{1}{|c}{\textbf{Time$_{\tiny{\mbox{MIP}}}$}} \\ 
  & \multicolumn{1}{|c|}{ (\%)} & \multicolumn{1}{|c|}{(sec)} &  &  & \multicolumn{1}{|c}{ (sec)} \\
\hline 
$\tilde {\mbox{LPR}}$                   & 38.082       & 0.35                     & 1.0         &                &                    \\ \hline
SCI$_1$                        & 36.703       & 32.38         & 1.1                  & 28.4           &                    \\ \hline
SCI           & 10.036       & 39.98                    & 3.0             & 364.1          &                    \\ \hline
CILS                     & 19.719       & 9.00                     & 2.7       & 82.2           & 6.91               \\ \hline 
SCILS                    & 9.121        & 266.81                & 50.0         & 794.3          & 198.12             \\ \hline
ALL           & 3.315        & 315.82                    & 28.3         & 646.6          & 264.91             \\ \hline
\end{tabular}
\caption{Impact of the cuts added to $\tilde {\mbox{LPR}}$  on 10 small instances ($n = 10$).}
\label{tableCuts10_100}
\end{table}

Figure \ref{gapsCuts10_100} depicts  the optimality gaps from Table  \ref{tableCuts10_100}. 
There is a  trade-off between the quality of the cuts and the computational time needed to find them. Considering a unique cut of each type, we  note that \hyperlink{scilsineq}{\SCILS} is the strongest cut (OptGap $= 9.121\%$), but  the computational time to obtain it, if compared to \hyperlink{cilsineq}{\CILS}  and \hyperlink{sciineq}{\SCIp}, is bigger. Nevertheless, a decrease in the times could be achieved with a heuristic solution for the separation problems, and also by the application of better stopping criteria for the cutting plane algorithm. 
We point out that using all cuts together we find a better upper bound than using each type of cut  separately (OptGap $= 3.315\%$). 

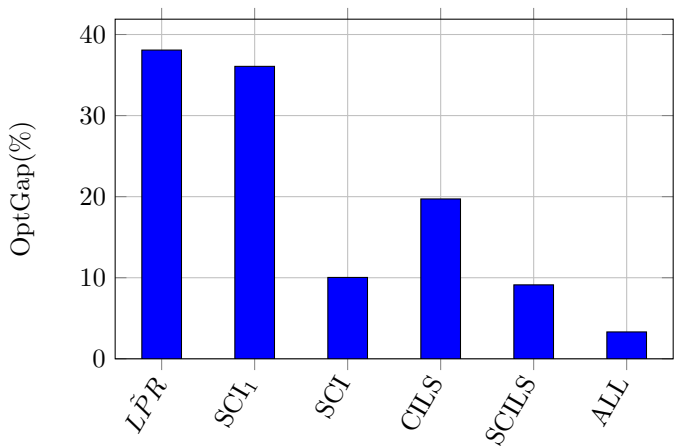
\begin{figure}[h]
\centering
\begin{tikzpicture}
        \begin{axis}[
            symbolic x coords={$\tilde{LPR}$, \text{SCI$_1$}, SCI, CILS, SCILS, ALL},
			ylabel={OptGap(\%)},    
			x tick label style={rotate=60,anchor=east},
            every node near coord/.append style={rotate=90, anchor=west},	
			grid=major,            
            width=9cm,
			height=6.1cm,
			ymin=0.0,
			ybar=0pt,
			bar shift=0pt,
			bar width=15,
			xtick=data         
          ]
            \addplot[ybar,fill=blue] coordinates {
                ($\tilde{LPR}$,   38.082)
                (\text{SCI$_1$},  36.073)
                (SCI,   10.036)
                (CILS,   19.719)
                (SCILS,   9.121)
                (ALL,   3.315)
            };
        \end{axis}        
\end{tikzpicture}
\caption{Average optimality gaps from Table \ref{tableCuts10_100}}
\label{gapsCuts10_100}
\end{figure}

We now analyze the effectiveness of our IPM in decreasing the upper bound while optimizing the perturbation $Q_p$. To improve the bounds obtained, besides the constraints in \eqref{linsdpc}, we  also consider in the initial relaxation,   the valid inequalities $X_{ii} = x_i$, the McCormick inequalities $X_{ij}\leq X_{ii}$,  and the valid inequalities obtained by multiplying the capacity constraint by each nonnegative variable $x_i$, and also by $(1-x_i)$, and then replacing each bilinear term $x_ix_j$ by $X_{ij}$. We then start the algorithm solving the following relaxation.

\begin{equation}
\label{linsdpb}
({\QPR}) \, \begin{array}{lll}
   \max &  x^T (Q-Q^{0}_p) x +\trace (Q^{0}_p X)\cr
 \mbox{s.t.}  & \sum_{j=1}^n w_j x_j \leq c ,\\
& \sum_{j=1}^n w_j X_{ij} \leq c X_{ii},&\forall i\in N\\
& \sum_{j=1}^n w_j (X_{jj} - X_{ij}) \leq c (1-X_{ii}),&\forall i\in N\\
 & X_{ii} = X_{ii},&\forall i\in N\\
 &X_{ij}\leq X_{ii}, &\forall i,j\in N\\
&0\leq X_{ij}\leq 1, &\forall i,j\in N\\
&0\leq x_{i}\leq 1, &\forall i\in N\\
&X\in \mathbb{S}^n.
\end{array}
\end{equation}

In order to evaluate the influence of the initial decomposition of $Q$ on the behavior of the IPM, we considered two initial decompositions.  In both cases, we  compute the eigendecomposition of $Q$, getting $Q=\sum_{i=1}^{n}\lambda_iv_iv_i'$. 

\begin{itemize}
\item For the first decomposition, we set  $Q_n:=\sum_{i=1}^{n}(-|\lambda_i|-1)v_iv_i'$, and $Q^{0}_p:= Q - Q_n/2$. We refer to this initial matrix $Q_p^0$ as $Q_p^a$.
\item For the second, we set  $Q_n:=\sum_{i=1}^{n}(\min\{\lambda_i,-10^{-6}\})v_iv_i'$, and $Q^{0}_p:= Q - Q_n/2$. We refer to this initial matrix $Q_p^0$ as $Q_p^b$.
\end{itemize}

In Table \ref{table50_100sdp}, we  compare the bounds obtained by our IPM after 20 iterations ($boundIPM_{20}$),  with the bounds given from the linear SDP relaxation obtained by taking $Q_p^{0}=Q$ in \eqref{linsdpb}, and adding to it the semidefinite constraint $X-xx'\succeq 0$ ($boundSDP$). As mentioned in Section \ref{sec:sdprel}, these are the best possible bounds that can be obtained by the IPM algorithm. We also show in Table \ref{table50_100sdp} how close to  $boundSDP$, the bound computed with the initial decomposition $Q_p^0$ ($boundIPM_1$) in relaxation \eqref{linsdpb} is.  The values presented in the table are
\begin{align*}
&gap_1  (\%)=(boundIPM_1-boundSDP)/boundSDP *100\\
&gap_{20}   (\%)=(boundIPM_{20}-boundSDP)/boundSDP *100
\end{align*}

\begin{table}[h]
\centering
\begin{tabular}{l|c|c|c|c}
\hline
\multicolumn{1}{l|}{$Q_p^0$}  & \multicolumn{1}{l|}{Inst}  & $n$    & \multicolumn{1}{|c|}{$gap_1$ (\%)} &  \multicolumn{1}{|c|}{$gap_2$ (\%)} \\ 
\hline
$Q_p^a$  & I1 &   50     & 0.30 & 0.01 \\
                        & I2 &    50   & 0.73  & 0.03\\
                        & I3 &    50    & 0.14 &  0.00\\
                        & I4 &    50   & 1.02 & 0.21\\
                        & I5 &    50  & 0.59 & 0.09  \\          
\cline{2-5}
  & I1 & 100      &  1.47    &0.14 \\
                        & I2 & 100   &  0.59   & 0.04\\
                        & I3 & 100  &  0.51    & 0.05 \\
                        & I4 & 100 &  1.38  & 0.26  \\
                        & I5 & 100&  0.73 &  0.06 \\       
\hline
$Q_p^b$  & I1 & 50   &  0.01   & 0.00  \\
                        & I2 & 50  &   0.30   &0.09 \\
                        & I3 & 50  &   0.08  &0.03  \\
                        & I4 & 50   &  0.10   & 0.02\\
                        & I5 & 50   &   0.03 & 0.03  \\           
\cline{2-5}
  & I1 & 100  &   0.04  & 0.00  \\
                        & I2 &100 & 0.02    & 0.01\\
                        & I3 & 100  &  0.03   &0.01  \\
                        & I4 & 100   & 0.12    &  0.04  \\
                        & I5 & 100   &0.02    & 0.01   \\    
\hline 
\end{tabular}
\caption{SDP bound vs  IPM bound at iterations 1 and 20, for two initial matrices $Q_p^0$.}
\label{table50_100sdp}
\end{table}

For the experiment reported in Table \ref{table50_100sdp}, we consider 10 instances with $n=50$ and $100$. We see from the results in Table \ref{table50_100sdp} that in 20 iterations, the IPM  closed the gap to the SDP bounds  for all instances. When starting from $Q_p^a$, we end up with an average  bound less than  0.1\% of the SDP bound,  while when starting from $Q_p^b$, this percentage decreases  to only 0.03\%. We also start from better bounds when considering  $Q_p^b$, and therefore, we use this matrix as the initial decomposition for the IPM in the next experiments.  The results in Table \ref{table50_100sdp} show that the IPM developed in this paper is effective to solve the parametric problem \eqref{eq:paramQQQ},  converging to bounds very close to  the solution of the SDP relaxation, which are their minimum possible values.

We finally present results obtained from the application of \verb;CWICS;,  considering the parametric quadratic relaxation,  the IPM,  and the cuts.
In Tables \ref{table50_100f} and \ref{table100_100f} we show the results for the same instances with $n=50$ and $n=100$ considered in the previous experiment.
 The cuts are added at every $m$ iterations of the IPM and the numbers of cuts added at each iteration are $n$ \hyperlink{sciineq}{\SCIp}, 5 \hyperlink{cilsineq}{\CILS} and 5 \hyperlink{scilsineq}{\SCILS}. 
Note that when solving each MIQP or MILP problem, besides the cut   \hyperlink{cilsineq}{\CILS} or \hyperlink{scilsineq}{\SCILS}, we also obtain a  cover inequality  \hyperlink{ciineq}{\CI}. We check if this \hyperlink{ciineq}{\CI} was already added to the  relaxation, and if not, we add it as well. We stop \verb;CWICS; when a maximum number of iterations is reached or when DuGap is sufficiently small. 


For the results presented in Table \ref{table50_100f} ($n=50$), we set the maximum number of iterations of the IPM equal to 100, and $m=10$.  The execution of each separation problem was limited to 3 seconds, and the best solutions obtained in this time limit was used to generate the cuts.

\begin{table}[h]
\centering
\begin{tabular}{c|c|r|c|c|c}
\hline
\multicolumn{1}{l|}{Inst}  &\multicolumn{1}{|c|}{OptGap }& \multicolumn{1}{|c|}{Time} &\multicolumn{1}{|c|}{DuGap } & \multicolumn{1}{|c|}{Iter}&  \multicolumn{1}{|c}{Time$_{\tiny{\mbox{MIP}}}$} \\ 
 & \multicolumn{1}{|c|}{ (\%)} & \multicolumn{1}{|c|}{(sec)} &  \multicolumn{1}{|c|}{ (\%)} &  & \multicolumn{1}{|c}{ (sec)} \\
\hline
 I1& 0.23	&1013.50	&0.27	&100		&641.98\\
I2&\textbf{0.00}	&632.50	&\textbf{0.00}&	64		&411.67\\
I3&                                \textbf{0.00}	&392.55	&\textbf{0.00}&	44 		&205.70\\
I4&                               \textbf{0.00}	&289.97	&\textbf{0.00}	&31		&160.37\\
I5&                               0.21	&1093.60	&0.37	&10	&698.04\\
\hline
\end{tabular}
\caption{Results for CWICS ($n=50$). }
\label{table50_100f}
\end{table}

For the results presented in Table \ref{table100_100f} ($n=100$), we set the maximum number of iterations of the IPM equal to 20, and $m=4$.  In this case, the execution of each separation problem was limited to 10 seconds. 

\begin{table}[h]
\centering
\begin{tabular}{c|c|c|c|c|c}
\hline
\multicolumn{1}{l|}{Inst}  &\multicolumn{1}{|c|}{OptGap }& \multicolumn{1}{|c|}{Time} &\multicolumn{1}{|c|}{DuGap } & \multicolumn{1}{|c|}{Iter}&  \multicolumn{1}{|c}{Time$_{\tiny{\mbox{MIP}}}$} \\ 
 & \multicolumn{1}{|c|}{ (\%)} & \multicolumn{1}{|c|}{(sec)} &  \multicolumn{1}{|c|}{ (\%)} &  & \multicolumn{1}{|c}{ (sec)} \\
\hline
I1 & \textbf{0.00}  & 2035.30 &\textbf{0.00} & 20       & 737.86\\
I2 & 0.25              & 2177.30 & 0.65 & 20       & 919.41\\
I3 & \textbf{0.00}  & 2007.10  &  \textbf{0.00} & 20       & 773.00 \\
I4 & 0.12               & 1885.90 & 0.84 & 20         & 828.98 \\
I5 & 0.04               & 2309.50 & 0.20 & 20        & 970.49 \\
\hline
\end{tabular}
\caption{Results for CWICS ($n=100$). }
\label{table100_100f}
\end{table}

We note from the results in Tables \ref{table50_100f} and \ref{table100_100f}, that the alternation between the iterations of the IPM to improve the perturbation $Q_p$ of the relaxation and the addition of cuts to the relaxation, changing the search direction of the IPM, is an effective approach to compute bounds for $\hyperlink{modelqkp}{\QKP}$. Considering the stopping criterion imposed to \verb;CWICS;, it was able to converge to the optimum solution of three out of five instances with $n=50$ and of two out of five instances with $n=100$. The average optimality gap for all ten instances is less than   $0.1\%$.  The heuristic applied also computed good solutions for the problem. The average duality gap for the 10 instances is less than $0.25\%$.

We note that our algorithm spends a high percentage of its running time solving the separation problems, and also solving the linear systems to define the direction of improvement in the IPM algorithm. The running time of both procedures can be improved by a more judicious implementation. There are two parameters in \verb;CWICS; that can also be better analyzed and tuned to improve the results, namely, $m$ and the time limit for the execution of the separation problems. As mentioned before, these problems could still be solved by heuristics. Finally, we note that the alternation between the IPM iterations and addition of cuts to the relaxation could be  combined with a branch-and-bound algorithm in an attempt to converge faster to the optimal solution. In this case, the cuts added to the relaxations would include the cuts that define the branching and the update on $Q_p$ would depend on the branch of the enumeration tree. These are directions for the continuity of the research on this work.

\section{Conclusion}
\label{sec:conclusion}
In this paper we present an algorithm named \verb;CWICS; (Convexification With Integrated Cut Strengthening),  that iteratively improves the upper bound for the quadratic knapsack problem (QKP). The initial relaxation for the problem is given by a parametric convex quadratic problem, where the Hessian $Q$ of the objective function of the QKP is perturbed by a matrix parameter $Q_p$, such that $Q-Q_p\preceq 0$. Seeking for the best possible bound, the concave term  $x^T(Q-Q_p)x$, is then kept in the objective function of the relaxation and the remaining part, given by  $x^TQ_px$ is linearized through the standard approach that lifts the problem to space of symmetric matrices defined by $X:=xx^T$. 

We present a primal-dual interior point method (IPM), which update the perturbation $Q_p$ at each iteration of \verb;CWICS;  aiming at reducing the upper bound given by the relaxation. We also present new classes of cuts that are added during the execution of \verb;CWICS;, which are defined on the lifted variable $X$, and derived from cover inequalities and the binary constraints.

We show that both the IPM and the cuts generated are effective in improving the upper bound for the QKP and apply them in an integrated procedure  that alternates between the optimization of the perturbation $Q_p$ and the strengthening of the feasible set of the relaxation with the addition of the cuts.   We note that these procedures could be applied to more general binary indefinite quadratic problems as well. The separation problems described to generate the cuts could also be solved heuristically, in order to accelerate the process.

We note that the search for the best perturbation $Q_p$, by our IPM, is updated with the inclusion of cuts to the relaxation. In the set of cuts added, we could also consider cuts defined by the branching procedure in a branch-and-bound algorithm. In this case, we could have the perturbation $Q_p$ optimized during all the descend on the branch-and-bound tree, considering the cuts that had been added to the relaxations. 

Finally, we show that if the positive semidefinite constraint $X-xx^T\succeq 0$ was introduced in the relaxation of the QKP, or any other  indefinite quadratic problem (maximizing the objective function), then the decomposition of objective function, that leads to a convex quadratic SDP relaxation, where a perturbed concave part of the objective is kept, and the remaining part is linearized, is not effective. In this case the best bound is always attained when the whole objective function is linearized, i.e., when the perturbation $Q_p$ is equal to $Q$. This observation also relates to the well known DC (difference of convex) decomposition of indefinite quadratics that have been used in the literature to generate bounds for indefinite quadratic problems. Once more, in case the positive semidefinite  constraint is added to the relaxation, the DC decomposition is not effective anymore, and the alternative linear SDP relaxation leads to the best possible bound.  As corollary from this result, we see that the bound given by the convex quadratic relaxation cannot be better than the bound given by the corresponding linear SDP relaxation. 

\section*{Acknowledgements}
We thank J. O. Cunha for providing us with the test instances used in our experiments. M. Fampa was partially supported by CNPq (National Council for Scientific and
Technological Development) Grant 303898/2016-0, D. C. Lubke was supported by Research Grants from CAPES (Brazilian Federal Agency for Support and Evaluation of Graduate Education) Grant 88881.131629/2016-01, CNPq Grant 142143/2015-4 and partially supported by PSR - Power Systems Research, Rio de Janeiro, Brazil. 


\bibliographystyle{plain}
\label{bibliog}
%
\def\cprime{$'$} \def\cprime{$'$} \def\cprime{$'$}
  \def\udot#1{\ifmmode\oalign{$#1$\crcr\hidewidth.\hidewidth
  }\else\oalign{#1\crcr\hidewidth.\hidewidth}\fi} \def\cprime{$'$}
  \def\cprime{$'$} \def\cprime{$'$}

\end{document}